\documentclass{amsart}
\usepackage[utf8]{inputenc}
\usepackage[hidelinks]{hyperref}
\usepackage{enumitem, stmaryrd, verbatim}

\newtheorem{corollary}{Corollary}
\newtheorem{lemma}{Lemma}
\newtheorem{proposition}{Proposition}
\newtheorem{theorem}{Theorem}

\theoremstyle{definition}
\newtheorem{definition}{Definition}

\theoremstyle{remark}
\newtheorem{remark}{Remark}

\DeclareMathOperator{\ad}{ad}
\DeclareMathOperator{\Aut}{Aut}
\DeclareMathOperator{\Der}{Der}
\DeclareMathOperator{\End}{End}
\DeclareMathOperator{\Hom}{Hom}
\DeclareMathOperator{\Inn}{InnDer}

\begin{document}

\title[The hom-associative Weyl algebras]{The hom-associative Weyl algebras in prime characteristic}
\author{Per B\"ack}
\address[Per B\"ack]{Division of Mathematics and Physics, The School of Education, Culture and Communication, M\"alar\-dalen  University,  Box  883,  SE-721  23  V\"aster\r{a}s, Sweden}
\email[corresponding author]{per.back@mdh.se}

\author{Johan Richter}
\address[Johan Richter]{Department of Mathematics and Natural Sciences, Blekinge Institute of Technology, SE-371 79 Karlskrona, Sweden}
\email{johan.richter@bth.se}


\begin{abstract}
We introduce the first hom-associative Weyl algebras over a field of prime characteristic as a generalization of the first associative Weyl algebra in prime characteristic. First, we study properties of hom-associative algebras constructed from associative algebras by a general ``twisting'' procedure. Then, with the help of these results, we determine the commuter, center, nuclei, and set of derivations of the first hom-associative Weyl algebras. We also classify them up to isomorphism, and show, among other things, that all nonzero endomorphisms on them are injective, but not surjective. Last, we show that they can be described as a multi-parameter formal hom-associative deformation of the first associative Weyl algebra, and that this deformation induces a multi-parameter formal hom-Lie deformation of the corresponding Lie algebra, when using the commutator as bracket.
\end{abstract}

\maketitle

\section{Introduction}

\emph{Hom-associative algebras} originate with the introduction of \emph{hom-Lie algebras}, the latter a class of algebras defined by Hartwig, Larsson, and Silvestrov~\cite{HLS03} to describe deformations of Lie algebras obeying a generalized Jacobi identity twisted by a \emph{hom}omorphism; hence the name. Hom-associative algebras, introduced by Makhlouf and Silvestrov~\cite{MS08}, play the same role as associative algebras do for Lie algebras; equipping a hom-associative algebra with the commutator as bracket give rise to a hom-Lie algebra. In a hom-associative algebra, the associativity condition is twisted by a homomorphism, similarly to the twisting of the Jacobi identity in hom-Lie algebras. By taking the homomorphism to be the identity map, one gets the class of associative algebras as a subclass of the hom-associative algebras. By taking the homomorphism to be the zero map, the twisted associativity condition becomes null, and hence non-associative algebras also constitute a subclass of hom-associative algebras.

The first associative Weyl algebra may be seen as an \emph{Ore extension}, or a \emph{non-commutative polynomial ring} as Ore extensions were first named by Ore when he introduced them~\cite{Ore33}. \emph{Non-associative Ore extensions} were later introduced by Nystedt, \"Oinert, and Richter in the unital case~\cite{NOR18}, and then generalized to the non-unital, hom-as\-so\-cia\-tive setting by Silvestrov and the authors~\cite{BRS18}. The authors further developed this theory in~\cite{BR18}, introducing a Hilbert's basis theorem for unital, non-associative and hom-associative Ore extensions.

In \cite{BRS18}, Silvestrov and the authors introduced the first hom-associative Weyl algebras over a field of characteristic zero, along with hom-associative versions of the quantum plane and the universal enveloping algebra of the two-dimensional non-abelian Lie algebra. In \cite{BR20} the authors studied the first hom-associative Weyl algebras in characteristic zero. They were shown to be a formal deformation of the associative Weyl algebra, and properties such as their center and their set of derivations were studied. The authors also proved a hom-associative analogue of the Dixmier conjecture to be true, which states that all nonzero endomorphisms on the first Weyl algebra in characteristic zero are automorphisms. In the characteristic zero case, all the first hom-associative Weyl algebras over the same field are isomorphic, except the associative one. 

In this paper, we introduce and study the first hom-associative Weyl algebras over a field of prime characteristic. Unlike the case of zero characteristic, there are many non-isomorphic ways of deforming the associative Weyl algebra to obtain hom-associative Weyl algebras. We classify the hom-associative Weyl algebras up to isomorphism, describe their commuter, center, and nuclei, as well as characterize the derivations on them. We also show how they can be described as a multi-parameter formal deformation of the associative Weyl algebra, among other things. 

This paper is organized as follows:

In \autoref{sec:preliminaries}, we gather preliminaries on non-associative algebras (\autoref{subsec:prel-non-assoc}), hom-associative algebras and hom-Lie algebras (\autoref{subsec:hom}), non-unital, hom-associative Ore extensions (\autoref{subsec:homOre}), and also on the first associative Weyl algebra (\autoref{subsec:weyl-algebra}) and on modular arithmetic (\autoref{subsec:modular-arithmetic}).

In \autoref{sec:yau-twist}, we introduce some lemmas on a particular construction of hom-associative algebras, called the Yau twist. These results are then used in the succeeding section to prove properties of the hom-associative Weyl algebras.

In \autoref{sec:morph-assoc-comm-der}, we define the hom-asso\-cia\-tive Weyl algebras and study their basic properties: we describe the possible homomorphisms that can twist the associativity condition and hence give rise to hom-associative Weyl algebras (\autoref{lem:alpha-sufficient-necessary}), we show that the hom-associative Weyl algebras contain no zero divisors (\autoref{cor:Weyl-algebra-zero-divisors}), that the commuter and the property of not being simple is the same as for the associative Weyl algebra (\autoref{cor:commuter} and \autoref{prop:non-simple}), and that the latter is the only hom-associative Weyl algebra which is power associative (\autoref{prop:power-assoc}). We also describe their nuclei (\autoref{prop:right-nucleus}), center (\autoref{cor:center}), and set of derivations (\autoref{prop:der-case-1} and \autoref{prop:der-case-2}), and moreover, show that every nonzero endomorphism is injective (\autoref{prop:hom-injective}), but that there are nonzero endomorphisms which are not surjective (\autoref{prop:hom-non-Dixmier}). Last, we classify all hom-associative Weyl algebras up to isomorphism (\autoref{prop:k-zero-isomorphism} and \autoref{prop:iso-necessary-sufficient}), and in many cases also describe all the possible isomorphisms (\autoref{prop:iso-necessary-sufficient}). 

In \autoref{sec:formal-deformation}, we first introduce the notions of multi-parameter formal hom-associative deformations (\autoref{def:multi-param-hom-assoc}) and multi-parameter formal hom-Lie deformations (\autoref{def:multi-param-hom-Lie}). We then show that the hom-associative Weyl algebras are a multi-parameter formal hom-associative deformation of the first associative Weyl algebra (\autoref{prop:hom-assoc-deform}), inducing a multi-parameter formal hom-Lie deformation of the corresponding Lie algebra, when using the commutator as bracket (\autoref{prop:hom-lie-deform}). 

\section{Preliminaries}\label{sec:preliminaries}
We denote by $\mathbb{N}$ the set of nonnegative integers and by $\mathbb{N}_{>0}$ the set of positive integers. $\mathbb{F}_{p^n}$ is the finite field of characteristic $p$ and cardinality $p^n$ for some prime $p$ and $n\in\mathbb{N}_{>0}$. If $K$ is a field, $K^\times$ is the multiplicative group of nonzero elements of $K$.
\subsection{Non-associative algebras}\label{subsec:prel-non-assoc} By a \emph{non-as\-so\-cia\-tive algebra} $A$ over an associative, commutative, and unital ring $R$, we mean an $R$-algebra which is not necessarily associative. Furthermore, $A$ is called \emph{unital} if there exists an element $1_A\in A$ such that for all $a\in A$, $a\cdot 1_A=1_A\cdot a=a$, and \emph{non-unital} if there does not necessarily exist such an element. For a non-unital, non-associative algebra $A$, we denote by $D_l(A)$ the set of left zero divisors of $A$, and by $D_r(A)$ the set of right zero divisors of $A$. The \emph{commutator} $[\cdot,\cdot]\colon A\times A\to A$ is defined by $[a,b]:=a\cdot b-b\cdot a$ for arbitrary $a,b\in A$, and the \emph{commuter} of $A$, written $C(A)$, as the set $C(A):=\{a\in A\colon [a,b]=0,\ b\in A\}$. The \emph{associator} $(\cdot,\cdot,\cdot)\colon A\times A\times A\to A$ is defined by $(a,b,c)=(a\cdot b)\cdot c-a\cdot (b\cdot c)$ for arbitrary elements $a,b,c\in A$, and the \emph{left}, \emph{middle}, and \emph{right nuclei} of $A$, denoted by $N_l(A), N_m(A),$ and $N_r(A)$, respectively, as the sets $N_l(A):=\{a\in A\colon (a,b,c)=0,\ b,c\in A\}$, $N_m(A):=\{b\in A\colon (a,b,c)=0,\ a,c\in A\}$, and $N_r(A):=\{c\in A\colon (a,b,c)=0,\ a,b\in A\}$. The \emph{nucleus} of $A$, written $N(A)$, is defined as the set $N(A):=N_l(A)\cap N_m(A)\cap N_r(A)$. The \emph{center} of $A$, denoted by $Z(A)$, is the intersection of the commuter and the nucleus, $Z(A):=C(A)\cap N(A)$. If the only two-sided ideals in $A$ are the zero ideal and $A$ itself, $A$ is called \emph{simple}. We can measure the non-associativity of $A$ by using the associator: $A$ is called \emph{power associative} if $(a,a,a)=0$, \emph{left alternative} if $(a,a,b)=0$, \emph{right alternative} if $(b,a,a)=0$, \emph{flexible} if $(a,b,a)=0$, and \emph{associative} if $(a,b,c)=0$ for all $a,b,c\in A$. Hence, if $A$ is not power associative, then $A$ is not left alternative, right alternative, flexible, or associative. An $R$-linear map $\delta\colon A\to A$ is called a \emph{derivation} if for any $a,b\in A$, $\delta(a\cdot b)=\delta(a)\cdot b+a\cdot\delta(b)$, and the set of  derivations of $A$ is denoted by $\Der_R(A)$. If $A$ is associative, then all maps of the form $\ad_a:=[a,\cdot]\colon A\to A$ for an arbitrary $a\in A$ are derivations, called \emph{inner derivations}, and the set of all inner derivations of $A$ are denoted by $\Inn_R(A)$. If $A$ is not associative, such a map need not be a derivation, however. Last, recall that $A$ \emph{embeds} into a non-unital, non-associative algebra $B$ if there is an injective homomorphism from $A$ to $B$, so that $A$ may be seen as a subalgebra of $B$.

\subsection{Hom-associative algebras and hom-Lie algebras}\label{subsec:hom}In this section, we recall some basic definitions and results concerning hom-associative algebras and hom-Lie algebras. Hom-associative algebras were first introduced in \cite{MS08}, and hom-Lie algebras in \cite{HLS03}, by starting from vector spaces. Here, we take the slightly more general approach used in e.g. \cite{BR18,BRS18,BR20}, replacing vector spaces by modules. As it turns out, most of the basic theory still hold in this latter, more general case. Now, let us state what we mean by a hom-associative algebra. 

\begin{definition}[Hom-associative algebra]
A \emph{hom-associative algebra} over an associative, commutative, and unital ring $R$, is a triple $(M,\cdot,\alpha)$, consisting of an $R$-module $M$, a binary operation $\cdot\colon M\times M \to M$ linear over $R$ in both arguments, and an $R$-linear map $\alpha\colon M\to M$, satisfying, for all $a,b,c\in M$, $\alpha(a)\cdot(b\cdot c)=(a\cdot b)\cdot \alpha(c)$. 	
\end{definition}

In the above definition, $\alpha$ is in a sense twisting the usual associativity condition, and hence it is referred to as a \emph{twisting map}. A hom-associative algebra is called \emph{multiplicative} if the twisting map is multiplicative, i.e. if it is an $R$-algebra homomorphism.  

\begin{remark}
A hom-associative algebra $(M,\cdot,\alpha)$ over an associative, commutative, and unital ring $R$, is in particular a non-unital, non-associative $R$-algebra, and if  $\alpha=\mathrm{id}_M$, a non-unital, associative  $R$-algebra. If $\alpha=0_M$, the hom-associativity condition becomes null, and hence hom-associative algebras can actually be considered as generalizations of both associative and non-associative algebras.	
\end{remark}

\begin{definition}[Hom-associative ring]
A \emph{hom-associative ring} is a hom-associative algebra over the integers.	
\end{definition}

\begin{definition}[Weakly unital hom-associative algebra]Let $A$ be a hom-associative algebra. If for all $a\in A$, $e_l\cdot a=\alpha(a)$ for some $e_l\in A$, we say that $A$ is \emph{weakly left unital} with \emph{weak left identity} $e_l$. In case $a\cdot e_r=\alpha(a)$ for some $e_r\in A$, $A$ is called \emph{weakly right unital} with \emph{weak right identity} $e_r$. If there is an $e\in A$ which is both a weak left and a weak right identity, $e$ is called a \emph{weak identity}, and $A$ \emph{weakly unital}.
\end{definition}

\begin{remark}The notion of a weak identity can thus be seen as a weakening of that of an identity. A weak identity, when it exists, need not be unique. The term weak unit has been used in the litterature for the same concept. 
\end{remark}

\begin{definition}[Homomorphism]
Let $R$ be an associative, commutative, and unital ring, and let $A$ and $B$ be two hom-associative $R$-algebras with twisting maps $\alpha$ and $\beta$, respectively. A \emph{homomorphism} from $A$ to $B$ is an $R$-linear and multiplicative map $f\colon A\to B$, satisfying  $f\circ\alpha=\beta\circ f$. If $A=B$, then $f$ is an \emph{endomorphism}.
\end{definition}

For any two hom-associative $R$-algebras $A$ and $B$, we denote by $\Hom_R(A,B)$ the set of homomorphisms from $A$ to $B$, and put $\End_R(A):=\Hom_R(A,A)$ for the set of endomorphisms. For convenience, we also introduce $C_{\End_R(A)}(\alpha,\beta):=\{f\in\End_R(A)\colon f\circ\alpha=\beta\circ f\}$, $C_{\End_R(A)}(\alpha):=C_{\End_R(A)}(\alpha,\alpha)$, and  $C_{\Der_R(A)}(\alpha):=\{\delta\in\Der_R(A)\colon \delta\circ\alpha=\alpha\circ\delta\}$. In case $A$ is associative, we view $A$ as a hom-associative algebra with twisting map $\mathrm{id}_A$. Hence $\Hom_R(A,B)$ and $\End_R(A)$ where $A$ and $B$ are associative coincide with the usual notation for homomorphisms and endomorphisms of associative algebras, respectively.

\begin{proposition}[\cite{Yau09}]\label{prop:star-alpha-mult} Let $R$ be an associative, commutative, and unital ring, $A$ an associative $R$-algebra, and $\alpha\in\End_R(A)$. Define a product $*\colon A\times A\to A$ by $a* b:=\alpha(a\cdot b)$ for all $a,b\in A$. Then $A^\alpha:=(A,*,\alpha)$ is a hom-associative $R$-algebra.
\end{proposition}

The construction of $A^\alpha$ from $A$ in \autoref{prop:star-alpha-mult} above was introduced by Yau in \cite{Yau09}, and is often referred to as the \emph{Yau twist of $A$}. We will use the subscript $*$ whenever the multiplication is that in $A^\alpha$; hence $[\cdot,\cdot]_*$ denotes the commutator in $A^\alpha$ and $(\cdot,\cdot,\cdot)_*$ denotes the associator in $A^\alpha$.

\begin{corollary}[\cite{FG09}]\label{lem:fregier-and-gohr}
Let $R$ be an associative, commutative, and unital ring, $A$ a unital, associative $R$-algebra with identity $1_A$, and $\alpha\in\End_R(A)$. Then $A^\alpha$ is weakly unital with weak identity $1_A$.	
\end{corollary}	

\begin{definition}[Hom-Lie algebra]\label{def:hom-lie} A \emph{hom-Lie algebra} over an associative, commutative, and unital ring $R$, is a triple $(M, [\cdot,\cdot], \alpha)$, consisting of an $R$-module $M$, a binary operation $[\cdot,\cdot]\colon M\times M\to M$ linear over $R$ in both arguments, and an $R$-linear map $\alpha\colon M\to M$, satisfying, for all $a,b,c\in M$, the following two axioms:
\begin{align*}
[a,a]&=0,\quad\text{(alternativity)},\\
\left[\alpha(a),[b,c]\right]+\left[\alpha(c),[a,b]\right]+\left[\alpha(b),[c,a]\right]&=0,\quad\text{(hom-Jacobi identity)}.\label{eq:hom-jacobi}
\end{align*}
\end{definition}
In the above definition, $[\cdot,\cdot]$ is called a \emph{hom-Lie bracket}, and just like in the definition of hom-associative algebras, $\alpha$ is called a twisting map. As in the case of Lie algebra, we immediately get anti-commutativity of the bracket from the bilinearity and alternativity:  $0=[a+b,a+b]=[a,a]+[a,b]+[b,a]+[b,b]=[a,b]+[b,a]$, so $[a,b]=-[b,a]$ holds for all $a$ and $b$ in a hom-Lie algebra as well. Unless $R$ has characteristic two, anti-commutativity also implies alternativity, since $[a,a]=-[a,a]$ for all $a$.

\begin{remark}If $\alpha=\mathrm{id}_M$ in \autoref{def:hom-lie}, we get the definition of a Lie algebra. Hence the notion of a hom-Lie algebra can be seen as a generalization of that of a Lie algebra.
\end{remark}

\begin{proposition}[\cite{MS08}]\label{prop:hom-assoc-commutator} Let $(M,\cdot,\alpha)$ be a hom-associative algebra over an associative, commutative, and unital ring $R$, with commutator $[\cdot,\cdot]$. Then $(M,[\cdot,\cdot],\alpha)$ is a hom-Lie algebra over $R$.
\end{proposition}

Note that when $\alpha$ is the identity map in the above proposition, we recover the classical construction of a Lie algebra from an associative algebra. 

\subsection{Non-unital, hom-associative Ore extensions}\label{subsec:homOre}
In this section, we give some preliminaries from the theory of non-unital, hom-associative Ore extensions, as introduced in~\cite{BRS18}. First, if $R$ is a  non-unital, non-associative ring, a map $\beta\colon R\to R$ is called \emph{left $R$-additive} if for all $r,s,t\in R$, $r\cdot\beta(s+t)=r\cdot \beta(s)+r\cdot\beta(t)$. If given two left $R$-additive maps $\delta$ and $\sigma$ on a non-unital, non-associative ring $R$, as a set, a \emph{non-unital, non-associative Ore extension} of $R$, written $R[x;\sigma,\delta]$, consists of all formal sums $\sum_{i\in\mathbb{N}}r_ix^i$ where finitely many $r_i\in R$ are nonzero. $R[x;\sigma,\delta]$ is then equipped with the same addition as a commutative polynomial ring,
\begin{equation}
\sum_{i\in\mathbb{N}}r_ix^i+\sum_{i\in\mathbb{N}}s_ix^i=\sum_{i\in\mathbb{N}}(r_i+s_i)x^i,\quad r_i,s_i\in R,
\end{equation}
but with the following multiplication, first defined on \emph{monomials} $rx^m$ and $sx^n$ where $m,n\in\mathbb{N}$:
\begin{equation}
rx^m\cdot sx^n=\sum_{i\in\mathbb{N}}\left(r\cdot\pi_i^m(s\right))x^{i+n},\label{eq:ore-mult}
\end{equation}
and then extended to arbitrary \emph{polynomials} $\sum_{i\in\mathbb{N}}r_ix^i$ in $R[x;\sigma,\delta]$ by imposing distributivity. The functions $\pi_i^m\colon R\to R$ are defined as the sum of all $\binom{m}{i}$ compositions of $i$ instances of $\sigma$ and $m-i$ instances of $\delta$, so $\pi_2^3=\sigma\circ\sigma\circ\delta+\sigma\circ\delta\circ\sigma+\delta\circ\sigma\circ\sigma$ and $\pi_0^0=\mathrm{id}_R$. Whenever $i<0$, or $i>m$, we put $\pi_i^m\equiv 0$. That this really gives an extension of the ring $R$, as suggested by the name, can now be seen by the fact that $rx^0\cdot sx^0=\sum_{i\in\mathbb{N}}(r\cdot \pi_i^0(s))x^{i+0}=(r\cdot\pi_0^0(s))x^0=(r\cdot s)x^0$, and similarly $rx^0+sx^0=(r+s)x^0$ for any $r,s\in R$. Hence the isomorphism $r\mapsto rx^0$ embeds $R$ into $R[x;\sigma,\delta]$. In this paper, we shall only be concerned with the case $\sigma=\mathrm{id}_R$, however, in which case \eqref{eq:ore-mult} simplifies to
\begin{equation}
rx^m\cdot sx^n=\sum_{i\in\mathbb{N}}\binom{m}{i}\left(r\cdot\delta^{m-i}(s)\right)x^{i+n}.
\end{equation}
Starting with a non-unital, non-associative ring $R$ equipped with two left $R$-additive maps $\delta$ and $\sigma$ and an additive map $\alpha\colon R\to R$, we \emph{extend $\alpha$ homogeneously} to $R[x;\sigma, \delta]$ by putting $\alpha(rx^m)=\alpha(r)x^m$ for all $rx^m\in R[x;\sigma,\delta]$, imposing additivity. If $\alpha$ is further assumed to be multiplicative and to commute with $\delta$ and $\sigma$, we can turn a non-unital (unital), associative Ore extension into a non-unital (weakly unital), hom-associative Ore extension by using this extension, as the following proposition demonstrates:

\begin{proposition}[\cite{BRS18}]\label{prop:hom*ore} Let $S:=R[x;\sigma,\delta]$ be a non-unital, associative Ore extension of a non-unital, associative ring $R$, and $\alpha\colon R\to R$ a ring endomorphism that commutes with $\delta$ and $\sigma$, and extend $\alpha$ homogeneously to $S$. Then $S^\alpha$ is a multiplicative, non-unital, hom-associative Ore extension.
\end{proposition}

\begin{remark}\label{re:weak-unit}If $S$ in \autoref{prop:hom*ore} is unital with identity $1_S$, then by by \autoref{lem:fregier-and-gohr}, $S^\alpha$ is weakly unital with weak identity $1_S$ .
\end{remark}

\subsection{The first associative Weyl algebra}\label{subsec:weyl-algebra}
Let $K$ be a field of prime characteristic $p$. The first Weyl algebra $A_1$ over $K$ is the free, unital, associative algebra on two letters $x$ and $y$, $K\langle x,y\rangle$, modulo the commutation relation $[x,y]:=x\cdot y - y\cdot x =1_{A_1}$. $A_1$ is a non-commutative domain, but it should be mentioned, however, that there is an alternative definition of $A_1$ as an algebra of differential operators, and as such, it is not a domain (see e.g. Chapter 2.3 in \cite{Cou95}). Throughout this paper, we will stick to the former definition of $A_1$, though. Revoy \cite{Rev73} has shown that $C(A_1)=K[x^p,y^p]$. Furthermore, $A_1$ is isomorphic to the unital, associative Ore extension $K[y][x;\mathrm{id}_{K[y]},\mathrm{d}/\mathrm{d}y]$ (see e.g. \cite{GW04} or \cite{MR01}, both also being nice introductions to the subject of associative Ore extensions), and as a vector space over $K$, it has a basis $\{y^mx^n\colon m,n\in\mathbb{N}\}$. As $A_1$ contains no zero divisors, any nonzero endomorphism $f$ on $A_1$ is unital, since $f(1_{A_1})=f(1_{A_1})\cdot f(1_{A_1})\iff f(1_{A_1})\cdot (1_{A_1}-f(1_{A_1}))=0\iff f(1_{A_1})=1_{A_1}$. Dixmier \cite{Dix68} first described the automorphism group $\Aut_K(A_1)$ of $A_1$ in characteristic zero, and later Makar-Limanov \cite{Mak84} generalized Dixmier's result to arbitrary characteristic, as follows:

\begin{theorem}[\cite{Mak84}]\label{thm:Mak84} $\Aut_K(A_1)$ is generated by \emph{linear automorphisms}, 
\begin{equation*}
x\mapsto ax+by,\quad y\mapsto cx+dy, \quad \begin{vmatrix}a&c\\b&d\end{vmatrix}=1,\quad a,b,c,d\in K,
\end{equation*}
and \emph{triangular automorphisms},
\begin{equation*}
x\mapsto x,\quad y\mapsto y+q(x),\quad q(x)\in K[x].
\end{equation*}
\end{theorem}

In characteristic zero, all derivations of $A_1$ are inner derivations. This follows for instance from Sridharan's \cite{Sri61} computation of cohomology groups of $A_1$ in positive degree, which all turn out to be zero (cf. Remark 6.2 and Theorem 6.1). In prime characteristic, not all derivations are inner, however (contrary to what is stated in \cite{Rev73}). In \cite{BLO15}, Benkart, Lopes, and Ondrus found two non-inner derivations, and described $\Der_K(A_1)$ as in the following theorem:

\begin{theorem}[\cite{BLO15}]\label{thm:BLO15}$\Der_K(A_1)=C(A_1)E_x\oplus C(A_1)E_y\oplus\Inn_K(A_1)$ where $E_x,E_y\in\Der_K(A_1)$ are defined by $E_x(x)=y^{p-1}$, $E_x(y)=E_y(x)=0$, $E_y(y)=x^{p-1}$.
\end{theorem}

\subsection{Modular arithmetic}\label{subsec:modular-arithmetic} We recall that if $p$ is a prime, any number $n\in\mathbb{N}$ can be written as $n_0+n_1p+n_2p^2+\cdots+n_jp^j$ for some nonnegative integers $n_0,n_1,n_2,\ldots,n_j<p$. This is often referred to as the \emph{base $p$ expansion} of $n$.  

\begin{theorem}[Lucas's Theorem \cite{Luc78}]\label{thm:lucas}Assume $p$ is a prime, $m,n\in\mathbb{N}$, and that $m_0+m_1p+m_2p^2+\cdots+m_jp^j$ and $n_0+n_1p+n_2p^2+\cdots+n_jp^j$ are the base $p$ expansions of $m$ and $n$, respectively. Then $\binom{m}{n}\equiv\prod_{i=0}^j\binom{m_i}{n_i}\pmod{p}$ where $\binom{m_i}{n_i}:=0$ if $m_i<n_i$.
\end{theorem}

\begin{corollary}\label{cor:lucas}
If $p$ is a prime and $m\in\mathbb{N}_{>0}$, then $\binom{m}{n}\equiv0\pmod{p}$ for all $n\in\mathbb{N}_{>0}$ with $n<m$ if and only if $m=p^q$ for some $q\in\mathbb{N}_{>0}$.
\end{corollary}

\begin{proof}
Assume $m=p^q$ for some $q\in\mathbb{N}_{>0}$, and let the base $p$ expansion of $m$ be $m_0+m_1p+m_2p^2+\cdots+m_qp^q$. That is, $m_0=m_1=m_2=\ldots=m_{q-1}=0$ and $m_q=1$. Any $n\in\mathbb{N}_{>0}$ such that $n<m$ can be written as $n=n_0+n_1p+n_2p^2+\cdots +n_{q-1}p^{q-1}$ where at least one of $n_0,n_1,n_2,\ldots,n_{q-1}$ is nonzero, say $n_i$. Then $\binom{m_i}{n_i}=\binom{0}{n_i}=0$. By  \autoref{thm:lucas}, $\binom{m}{n}\equiv 0\pmod{p}$.

Assume $m$ is not a power of $p$. We wish to show that there is some $n\in\mathbb{N}_{>0}$ with $n<m$ such that $\binom{m}{n}$ is not divisible by $p$. To this end, let $m_0+m_1p+m_2p^2+\cdots+m_rp^r$ be the base $p$ expansion of $m$ where $m_r\neq0$. If $m_r=1$, then there is some $j<r$ such that $m_j\neq0$, since otherwise $m=p^r$. If $n:=m_0+m_1p+m_2p^2+\cdots+m_jp^j<m$, then by \autoref{thm:lucas}, $\binom{m}{n}\equiv\binom{1}{0}\binom{m_{r-1}}{0}\cdots\binom{m_j}{m_j}\cdots \binom{m_0}{m_0}\pmod{p}=1$. If $m_r>1$, then, with $n:=m_0+m_1p+m_2p^2+\cdots+(m_r-1)p^r<m$, \autoref{thm:lucas} gives $\binom{m}{n}\equiv\binom{m_r}{m_r-1}\binom{m_{r-1}}{m_{r-1}}\cdots\binom{m_0}{m_0}\pmod{p}=m_r<p$.
\end{proof}

\section{Seven little lemmas on Yau twisted algebras}\label{sec:yau-twist}
In this section, we introduce some lemmas describing properties of Yau twisted algebras in terms of properties of the underlying associative algebras. These results will then be used in proving properties of the first hom-associative Weyl algebras in the succeeding section. Throughout this section, $R$ is an associative, commutative, and unital ring.

\begin{lemma}\label{lem:unique-weak-unit}Let $A$ be a unital, associative $R$-algebra with identity $1_A$, and let $\alpha\in\End_R(A)$ be injective. Then $1_A$ is a unique weak left identity and a unique weak right identity in $A^\alpha$.
\end{lemma}

\begin{proof}Let $A$ be a unital, associative $R$-algebra with identity $1_A$, and assume that $\alpha\in\End_R(A)$ is injective. By \autoref{lem:fregier-and-gohr}, $1_A$ is a weak identity in $A^\alpha$. Assume $e_l$ is a weak left identity. Then $\alpha(e_l)=e_l*1_A=\alpha(1_A)$, so $e_l=1_A$ by the injectivity of $\alpha$. Similarly, if $e_r$ is a weak right identity, then $\alpha(e_r)=1_A*e_r=\alpha(1_A)$, so $e_r=1_A$.
\end{proof}

\begin{lemma}\label{lem:zero-divisors}Let $A$ be a non-unital, associative $R$-algebra, and let $\alpha\in\End_R(A)$. Then $D_l(A)\subseteq D_l(A^\alpha)$ and $D_r(A)\subseteq D_l(A^\alpha)$, with equality if $\alpha$ is injective.
\end{lemma}

\begin{proof}
We show the left case; the right case is similar. Let $A$ be a non-unital, associative $R$-algebra, and let $\alpha\in\End_R(A)$, $a\in D_l(A)$, and $b\in A^\alpha$. Then $a*b=\alpha(a\cdot b)=\alpha(0)=0$, so $a\in D_l(A^\alpha)$, and hence $D_l(A)\subseteq D_l(A^\alpha)$.

Now, assume that $\alpha$ is injective, $c\in D_l(A^\alpha)$ and $d\in A$. Then $0=c*d=\alpha(c\cdot d)\iff c\cdot d=0$, so $c\in D_l(A)$, and hence $D_l(A^\alpha)\subseteq D_l(A)$.
\end{proof}

\begin{lemma}\label{lem:commuter}
Let $A$ be a non-unital, associative $R$-algebra, and let $\alpha\in\End_R(A)$. Then $C(A)\subseteq C(A^\alpha)$, with equality if $\alpha$ is injective.	
\end{lemma}

\begin{proof}
Let $A$ be a non-unital, associative $R$-algebra, and let $\alpha\in\End_R(A)$, $a\in C(A)$ and $b\in A^\alpha$. Then $[a,b]_*=\alpha([a,b])=\alpha(0)=0$, so $a\in C(A^\alpha)$, and hence $C(A)\subseteq C(A^\alpha)$.

Now, assume that $\alpha$ is injective, $c\in C(A^\alpha)$ and $d\in A$. Then $\alpha([c,d])=[c,d]_*=0\iff [c,d]=0$, so $c\in C(A)$, and hence $C(A^\alpha)\subseteq C(A)$.
\end{proof}

\begin{lemma}\label{lem:alpha-beta-morphism-commute}
Let $A$ be a non-unital, associative $R$-algebra, and let $\alpha,\beta\in\End_R(A)$. Then $C_{\End_R(A)}(\alpha,\beta)\subseteq\Hom_R(A^\alpha,A^\beta)$, with equality if $\beta$ is injective.
\end{lemma}

\begin{proof}
Let $A$ be a non-unital, associative $R$-algebra, and let $\alpha, \beta\in\End_R(A)$ and $f\in C_{\End_{R}(A)}(\alpha,\beta)$. Denote by $*_\alpha$ the multiplication in $A^\alpha$, and by $*_\beta$ the multiplication in $A^\beta$. Then, for any $a,b\in A$, $f(a*_\alpha b)=f(\alpha(a\cdot b))=\beta(f(a\cdot b))=\beta(f(a)\cdot f(b))=f(a)*_\beta f(b)$. Since $f$ is $R$-linear, multiplicative, and satisfies $f\circ \alpha=\beta\circ f$ by assumption, it follows that $f\in\Hom_R(A^\alpha,A^\beta)$, and hence $C_{\End_R(A)}(\alpha,\beta)\subseteq\Hom_R(A^\alpha,A^\beta)$.

Now, assume that $\beta$ is injective and $g\in\Hom_R(A^\alpha,A^\beta)$. Then, for any $c,d\in A$, $g(c*_\alpha d)=g(\alpha(c\cdot d))=\beta(g(c\cdot d))$. On the other hand, $g(c*_\alpha d)=g(c)*_\beta g(d)=\beta(g(c)\cdot g(d))$, so by the injectivity of $\beta$, $g(c\cdot d)=g(c)\cdot g(d)$. By assumption, $g$ is $R$-linear and $g\circ\alpha=\beta\circ g$, so $g\in C_{\End_R(A)}(\alpha,\beta)$, and hence $\Hom_R(A^\alpha,A^\beta)\subseteq C_{\End_R(A)}(\alpha,\beta)$.
\end{proof}

By setting $\alpha=\beta$ in \autoref{lem:alpha-beta-morphism-commute}, we have the following lemma:

\begin{lemma}\label{lem:alpha-injective-endomorphisms-1}
Let $A$ be a non-unital, associative $R$-algebra, and let $\alpha\in\End_R(A)$. Then $C_{\End_R(A)}(\alpha)\subseteq \End_R(A^\alpha)$, with equality if $\alpha$ is injective.
\end{lemma}

\begin{lemma}\label{lem:derivations-subset}
Let $A$ be a non-unital, associative $R$-algebra, and let $\alpha\in\End_R(A)$. Then $C_{\Der_R(A)}(\alpha)\subseteq \Der_R(A^\alpha)$.
\end{lemma}

\begin{proof}
Let $A$ be a non-unital, associative $R$-algebra, and let $\alpha\in\End_R(A)$, $\delta\in C_{\Der_R(A)}(\alpha)$, and $a,b\in A^\alpha$. Then $\delta(a*b)=\delta(\alpha(a\cdot b))=\alpha(\delta(a\cdot b))=\alpha(\delta(a)\cdot b+a\cdot\delta(b))=\alpha(\delta(a)\cdot b)+\alpha(a\cdot\delta(b))=\delta(a)*b+a*\delta(b)$, so $\delta\in\Der_R(A^\alpha)$.	
\end{proof}

\begin{lemma}\label{lem:derivations}
Let $A$ be a unital, associative $R$-algebra with identity $1_A$, and let $\alpha\in\End_R(A)$ be injective. Then $\Der_R(A^\alpha)=C_{\Der_R(A)}(\alpha)$ if and only if $\delta(1_A)=0$ for any $\delta\in\Der_R(A^\alpha)$. 
\end{lemma}

\begin{proof}
Let $A$ be a unital, associative  $R$-algebra with identity $1_A$, and let $\alpha\in\End_R(A)$ be injective and $\delta\in\Der_R(A^\alpha)$. From \autoref{lem:derivations-subset}, $C_{\Der_R(A)}(\alpha)\subseteq \Der_R(A^\alpha)$. Assume $ \Der_R(A^\alpha)\subseteq C_{\Der_R(A)}(\alpha)$. Then, $\delta(1_A)=\delta(1_A\cdot 1_A)=\delta(1_A)\cdot 1_A+\delta(1_A)\cdot 1_A=2\delta(1_A)\iff \delta(1_A)=0$. 

Now, assume instead that $\delta(1_A)=0$. Then, for any $a\in A$, $\delta(\alpha(a))=\delta(1_A*a)=\delta(1_A)*a+1_A*\delta(a)=1_A*\delta(a)=\alpha(\delta(a))$. Hence, for any $b,c\in A$, $\alpha(\delta(b\cdot c))=\delta(\alpha(b\cdot c))=\delta(b*c)=\delta(b)*c+b*\delta(c)=\alpha(\delta(b)\cdot c+b\cdot\delta(c)$, and since $\alpha$ is injective, $\delta(b\cdot c)=\delta(b)\cdot c+b\cdot\delta(c)$. We can thus conclude that $\delta\in C_{\Der_R(A)}(\alpha)$, so that $\Der_R(A^\alpha)\subseteq C_{\Der_R(A)}(\alpha)$.
\end{proof}

\section{Morphisms, derivations, commutation and association relations}\label{sec:morph-assoc-comm-der}
In this section, we first single out the possible endomorphisms that may twist the associativity condition of $A_1$, and then use these endomorphisms to define the first hom-associative Weyl algebras as Yau twists of $A_1$. With the help of the results from the previous section, we then determine some basic properties of these Yau twisted Weyl algebras. 

\begin{lemma}\label{lem:alpha-sufficient-necessary}Let $\alpha$ be a nonzero endomorphism on $K[y]$. Then $\alpha$ commutes with $\mathrm{d}/\mathrm{d}y$ if and only if $\alpha(y)=k_0+y+k_py^p+k_{2p}y^{2p}+\cdots$, where only finitely many $k_0,k_p,k_{2p},\ldots\in K$ are nonzero.
\end{lemma}

\begin{proof}Let $\alpha$ be a nonzero endomorphism on $K[y]$. As $K[y]$ contains no zero divisors, $\alpha$ is unital (see \autoref{subsec:weyl-algebra}). Now, assume $\alpha(y)=k_0+k_1y+k_2y^2+\cdots,$ where only finitely many $k_0, k_1, k_2,\ldots\in K$ are nonzero, and put $\delta:=\mathrm{d}/\mathrm{d}y$. Then $\alpha(\delta(y))=\alpha(1_{K[y]})=1_{K[y]}$, and $\delta(\alpha(y))=k_1+2k_2y+3k_3y^2+\cdots$. Hence $\alpha(\delta(y))=\delta(\alpha(y))$ if and only if $\alpha(y)=k_0+y+k_py^p+k_{2p}y^{2p}+\cdots$. We claim that this is also a sufficient condition for $\alpha$ to commute with $\delta$ on $K[y]$. Now, as $\alpha$ and $\delta$ are both linear, we only need to verify that $\alpha(\delta(y^m))=\delta(\alpha(y^m))$ for all $m\in\mathbb{N}$. We prove this by induction over $m$. The case $m=0$ holds since $\delta\left(\alpha\left(1_{K[y]}\right)\right)=\delta\left(1_{K[y]}\right)=0$ and $\alpha\left(\delta\left(1_{K[y]}\right)\right)=\alpha(0)=0$, and the case $m=1$ was proven above. Continuing, assume that $\alpha(\delta(y^m))=\delta(\alpha(y^m))$ for some $m\in\mathbb{N}$. Then, $\delta\left(\alpha(y^{m+1})\right)=\delta\left(\alpha(y^m)\cdot\alpha(y)\right)=\delta(\alpha(y^m))\cdot\alpha(y)+\alpha(y^m)\cdot\delta(\alpha(y))=\alpha(\delta(y^m))\cdot\alpha(y)+\alpha(y^m)\cdot\alpha(\delta(y))=\alpha(\delta(y^m)\cdot y + y^m\cdot\delta(y))=\alpha(\delta(y^{m+1}))$.
\end{proof}

By applying \autoref{lem:alpha-sufficient-necessary} together with \autoref{prop:hom*ore} (and \autoref{re:weak-unit}) to $A_1$, we may now define some (weakly unital) hom-associative Ore extensions.

\begin{definition}[The hom-associative Weyl algebras]Let $k_0,k_p, k_{2p},\ldots\in K$ where only finitely many of the elements are nonzero, and put $k:=(k_0,k_p, k_{2p}, \ldots)$. Name the map in \autoref{lem:alpha-sufficient-necessary} $\alpha_{k}$, emphasizing its dependence on $k$. The \emph{hom-associative Weyl algebras} are the weakly unital, hom-associative Ore extensions $A_1^{\alpha_k}:=(A_1,*,\alpha_k)$. 
\end{definition}

For convenience, instead of $A_1^{\alpha_k}$, we shall write $A_1^k$.

\begin{remark}The associative Weyl algebra $A_1$ is the hom-associative Weyl algebra $A_1^0$. Note that as vector spaces over $K$, all the $A_1^k$ are the same.
\end{remark}

We denote by $\deg(q)$ the total degree of a polynomial $q\in A_1$, and set $\deg(0)=-\infty$.

\begin{lemma}\label{lem:alpha-injective}
	 $\alpha_k$ is injective.
\end{lemma}

\begin{proof}Let $q\in A_1$ be arbitrary. Then  $q=\sum_{i\in\mathbb{N}} q_i(y)x^i$ for some $q_i(y)\in K[y]$, so $\alpha_k(q)=\sum_{i\in\mathbb{N}}q_i(\alpha_k(y))x^i$. Assume that $\alpha_k(q)=0$. Then $-\infty=\deg(\alpha_k(q))=\deg(\alpha_k(y))\deg(q)$, and since $\deg(\alpha_k(y))>0$, we must have $\deg(q)=-\infty$. Hence $q=0$, so $\alpha_k$ is injective.\end{proof}

\begin{lemma}$\alpha_k$ is surjective if and only if $k=(k_0,0,0,0,\ldots)$.
\end{lemma}\label{lem:alpha-surjective}

\begin{proof}
Assume that $\alpha_k$ is surjective. Then $y=\alpha_k(q)$ for some $q=\sum_{i\in\mathbb{N}} q_i(y)x^i$ where $q_i(y)\in K[y]$. As $1=\deg(y)=\deg(\alpha_k(q))=\deg(\alpha_k(y))\deg(q)$, we must have $\deg(\alpha_k(y))=1$. Hence $\alpha_k(y)=k_0+y$, so that $k=(k_0,0,0,0,\ldots)$. 

Now, assume that $k=(k_0,0,0,0,\ldots)$. Then $\alpha_k$ is the endomorphism on $A_1$ defined by $\alpha_k(x)=x$ and $\alpha_k(y)=k_0+y$. By \autoref{thm:Mak84}, $\alpha_k$ is then a triangular automorphism, and hence surjective.
\end{proof}

\begin{proposition}\label{prop:subfield-of-hom-weyl} $K$ embeds as a subfield into $A_1^k$.
\end{proposition}

\begin{proof}The proof is identical to that in \cite{BR20} when $K$ has characteristic zero. Nevertheless, we provide it here for the convenience of the reader. $K$ embeds into the associative Weyl algebra $A_1:=K[y][x;\mathrm{id}_{K[y]},\mathrm{d}/\mathrm{d}y]$ by the isomorphism $f\colon K\to  K':=\{ay^0x^0\colon a\in K\}\subseteq A_1$ defined by $f(a)=ay^0x^0$ for any $a\in K$. One readily verifies that the same map embeds $K$ into $A_1^k$, i.e. it is also an isomorphism of the hom-associative algebra $K$, the twisting map being $\mathrm{id}_K$, and the hom-associative subalgebra $K'\subseteq A_1^k$.
\end{proof}

We define the operators $\frac{\partial}{\partial x}$ and $\frac{\partial }{\partial y}$ in the obvious way on $K[x]$ and $K[y]$, respectively, and then extend to elements $q=\sum_{i\in\mathbb{N}} q_i x^i, \ q_i \in K[y]$
by $\frac{\partial}{\partial x} q  = \sum_{i\in\mathbb{N}} i q_i(y) x^{i-1}$ and $\frac{\partial }{\partial y} q = \sum_{i\in\mathbb{N}} \left(\frac{\partial}{\partial y} q_i\right) x^i.$ Here, we set $0x^{-1}:=0$ and $0y^{-1}:=0$.

\begin{lemma}\label{lem:commutator-derivative}For any polynomial $q(x,y)\in A_1^k$,
\begin{align*}
[x,q(x,y)]_*=&\frac{\partial}{\partial y}q\left(x,k_0+y+k_py^p+ k_{2p}y^{2p}+\cdots\right),\\
[q(x,y),y]_*=&\frac{\partial}{\partial x}q\left(x,k_0+y+k_py^p+ k_{2p}y^{2p}+\cdots\right).
\end{align*}
\end{lemma}

\begin{proof}The proof is quite similar to the case when $K$ has characteristic zero (see Corollary 2 in \cite{BR20}). In $A_1$, $[x,y^mx^n]=x\cdot y^mx^n-y^mx^n\cdot x = \sum_{i\in\mathbb{N}}\binom{1}{i}\frac{\partial^{1-i}y^m}{\partial y^{1-i}}x^{n+i}-y^mx^{n+1}=my^{m-1}x^n$ for any $m,n\in\mathbb{N}$, where we define $0y^{-1}$ to be zero. Now, $[x,y^mx^n]_*=\alpha_k\left([x,y^mx^n]\right)=\alpha_k\left(my^{m-1}x^n\right)=m(k_0+y+k_py^p+k_{2p}y^{2p}+\cdots)^{m-1}x^n$, and so by using that $[\cdot,\cdot]_*$ is linear in the second argument, $[x,q(x,y)]_*=\frac{\partial}{\partial y}q\left(x,k_0+y+k_py^p+k_{2p}y^{2p}+\cdots\right)$. Similarly, $[y^mx^n,y]=y^mx^n\cdot y - y\cdot y^mx^n=\sum_{i\in\mathbb{N}}\binom{n}{i}\frac{\partial^{n-i}y}{\partial y^{n-i}}x^i-y^{m+1}x^n=ny^mx^{n-1}$ for any $m,n\in\mathbb{N}$ with $0x^{-1}$ defined to be zero. Hence $[y^mx^n,y]_*=\alpha_k\left([y^mx^n,y]\right)=\alpha_k\left(ny^mx^{n-1}\right)=n(k_0+y+k_py^p+k_{2p}y^{2p}+\cdots)^{m}x^{n-1}$. Using the linearity of $[\cdot,\cdot]_*$ in the first argument, $[q(x,y),y]_*=\frac{\partial}{\partial x}q\left(x,k_0+y+k_py^p+ k_{2p}y^{2p}+\cdots\right)$.
\end{proof}

\begin{lemma}$1_{A_1}$ is a unique weak left identity and a unique weak right identity in $A_1^k$.
\end{lemma}

\begin{proof}Since $\alpha_k$ is injective, this follows from \autoref{lem:unique-weak-unit}.
\end{proof}

\begin{corollary}\label{cor:Weyl-algebra-zero-divisors}There are no zero divisors in $A_1^k$.
\end{corollary}

\begin{proof}Since $\alpha_k$ is injective and since there are no zero divisors in $A_1$, by \autoref{lem:zero-divisors} there are no zero divisors in $A_1^k$.
\end{proof}

\begin{corollary}\label{cor:commuter}$C(A_1^k)=K[x^p,y^p]$.
\end{corollary}

\begin{proof}We know that $C(A_1)=K[x^p,y^p]$. Since $\alpha_k$ is injective,  $C(A_1^k)=C(A_1)$ by \autoref{lem:commuter}. 
\end{proof}

In characteristic zero, $A_1^k$ is simple (cf. Proposition 5.14 in \cite{BRS18}). The next proposition demonstrates that in prime characteristic, this is not the case.

\begin{proposition}\label{prop:non-simple}
$A_1^k$ is not simple.	
\end{proposition}

\begin{proof}
By \autoref{cor:commuter}, $x^p\in C(A_1^k)$, so the left ideal $I$ of $A_1^k$ generated by $x^p$ is also a right ideal. If $i\in I$ is nonzero, then $i=q*x^p=\alpha_k(q\cdot x^p)=\alpha_k(q)\cdot x^p$ for some nonzero $q$, so $\deg(i)=\deg(\alpha_k(q))+p\geq p$. Hence $I$ does not contain all elements of $A_1^k$, and since $I$ is nonzero, $A_1^k$ is not simple.
\end{proof}

\begin{lemma}\label{lem:associator}
For any $q,r,s \in A_1^k$,
$(q,r,s)_* =0 \iff q \cdot \alpha_k(r\cdot s) = \alpha_k(q\cdot r)\cdot s$.
\end{lemma}

\begin{proof}
 $(q,r,s)_* = q*(r*s)-(q*r)*s = \alpha_k(q\cdot \alpha_k(r\cdot s))-\alpha_k(\alpha_k(q\cdot r) \cdot s) = \\
 \alpha_k(q\cdot \alpha_k(r\cdot s) -\alpha_k(q\cdot r)\cdot s)$.
Since $\alpha_k$ is injective, the lemma follows. 
\end{proof}

\begin{proposition}\label{prop:power-assoc}
$A_1^k$ is power associative if and only if $k=0$.	
\end{proposition}

\begin{proof}If $k=0$, then $A_1^k$ is associative, and hence also power associative. To show the converse, we note that it follows from \autoref{lem:associator} that $(yx,yx,yx)_*=0$ if and only if $yx \cdot \alpha_{k}((yx)\cdot (yx)) = \alpha_{k}((yx)\cdot (yx)) \cdot yx$. But 
\begin{align*}
&&& yx \cdot \alpha_{k}((yx)\cdot (yx)) = \alpha_{k}((yx)\cdot (yx)) \cdot yx &\\
&&\iff &yx \cdot \alpha_{k}(y^2x^2+yx) = \alpha_{k}(y^2x^2+yx)\cdot yx \\ 
&& \iff &yx \cdot (\alpha_k(y)^2\cdot x^2+\alpha_k(y)\cdot x) = (\alpha_{k}(y)^2\cdot x^2+\alpha_{k}(y)\cdot x)\cdot yx \\
&& \iff &y\cdot\alpha_{k}(y)^2\cdot x^3+2y\cdot\alpha_{k}(y)\cdot x^2+y\cdot\alpha_{k}(y)\cdot x^2+yx \\
&&&= y\cdot\alpha_{k}(y)^2\cdot x^3+2\alpha_{k}(y)^2\cdot x^2+y\cdot\alpha_{k}(y)\cdot x^2+\alpha_k(y)\cdot x \\
&& \iff & 2y\cdot\alpha_{k}(y)\cdot x^2+yx = 2\alpha_{k}(y)^2\cdot x^2 +\alpha_{k}(y)\cdot x.
\end{align*}
The last equality clearly holds only if $k=0$. 
\end{proof}

From the above result, we can also conclude that $A_1^k$ is left alternative, right alternative, flexible, and associative if and only if $k=0$ (cf. \autoref{subsec:prel-non-assoc}).

\begin{proposition}\label{prop:right-nucleus}
$N_l(A_1^k) = N_m(A_1^k) = N_r(A_1^k)=\{0\}$ if and only if $k\neq0$.
\end{proposition}

\begin{proof}
If $k=0$, then $A_1^k$ is associative, so $N_l(A_1^k)=N_m(A_1^k)=N_r(A_1^k)=A_1^k$. Now, assume $k\neq0$. Let $r = \sum_{i\in\mathbb{N}} r_i x^i$ be some element in $N_r(A_1^k)$ where $r_i \in K[y]$. Then, $(1_{A_1},1_{A_1},r)_*=0$. By \autoref{lem:associator}, this is equivalent to $\alpha_k(r) = r\iff\alpha_k(r_i)=r_i$ for all $i\in\mathbb{N}$, so $r_i\in K$. We also have $(y,1_{A_1},r)=0$, which is equivalent to $y\cdot\alpha_k(r)=\alpha_k(y)\cdot r$. We first rewrite the right-hand side of the preceding equation: $\alpha_k(y) \cdot r = \sum_{i\in\mathbb{N}} r_i\alpha_k(y)x^i$. Then we rewrite the left-hand side: 
$y \cdot \alpha_k(r)=y\cdot r= \sum_{i\in\mathbb{N}} r_iyx^i$. Hence, for all $i\in\mathbb{N}$, it must be true that 
$r_i\cdot\alpha_k(y)=r_iy$. Since $k\neq0$ implies $\alpha_k(y)\neq y$, we have $r_i=0$ for all $i\in\mathbb{N}$. In other words, $r=0$, so $N_r(A_1^k)\subseteq\{0\}$. Since $\{0\}\subseteq N_r(A_1^k)$, we have $N_r(A_1^k)=\{0\}$. 

Assume $k\neq0$ and let $s=\sum_{i\in\mathbb{N}}s_ix^i$ be some element in $N_m(A_1^k)$ where $s_i\in K[y]$. Then, $(y,s,1_{A_1})_*=0$. By \autoref{lem:associator}, this is equivalent to $y\cdot\alpha_k(s)=\alpha_k(y)\cdot\alpha_k(s)$. This can be rewritten as $y\cdot\alpha_k(s_i)=\alpha_k(y)\cdot\alpha_k(s_i)\iff(y-\alpha_k(y))\cdot\alpha_k(s_i)=0$ for all $i\in\mathbb{N}$. Again, $k\neq0$ implies $\alpha_k(y)\neq y$, and by \autoref{cor:Weyl-algebra-zero-divisors}, there are no zero divisors in $A_1^k$. Hence $\alpha_k(s_i)=0$ for all $i\in\mathbb{N}$. By \autoref{lem:alpha-injective}, $\alpha_k$ is injective, and so $s_i=0$ for all $i\in\mathbb{N}$, and $N_m(A_1^k)=\{0\}$. 

By similar calculations, $N_l(A_1^k)=\{0\}$.
\end{proof}

\begin{corollary}$N(A_1^k)=\{0\}$ if and only if $k\neq0$.	
\end{corollary}

\begin{corollary}\label{cor:center}$Z(A_1^k)=\{0\}$ if and only if $k\neq0$.	
\end{corollary}

\begin{lemma}\label{lem:derivation-unit}
The following statements are equivalent:
\begin{enumerate}[label=(\roman*)]
	\item $\delta\in\Der_K(A_1^k)$.\label{it:derOne}
	\item $\delta\in C_{\Der_K(A_1)}(\alpha_k)$.\label{it:derTwo}
	\item $\delta\in\Der_K(A_1)$ satisfying $\alpha_k(\delta(x))=\delta(x)$ and $\alpha_k(\delta(y))=\delta(y)+k_p\delta(y^p)+k_{2p}\delta(y^{2p})+\cdots$.\label{it:derThree}
	\item $\delta=uE_x+vE_y+\ad_q$, where $E_x,E_y\in \Der_K(A_1)$ are defined by $E_x(x)=y^{p-1}$, $E_x(y)=E_y(x)=0$, $E_y(y)=x^{p-1}$, and $u,v\in K[x^p,y^p]$ and $q\in A_1$ satisfy $\alpha_k(u)=u$, $\alpha_k(v)=v$, $\frac{\partial}{\partial x}\left(\alpha_k(q)-q\right)=v\cdot\frac{\mathrm{d}}{\mathrm{d}(y^p)}(y-\alpha_k(y))$, and $\frac{\partial}{\partial y}\left(\alpha_k(q)-q\right)=u\cdot(\alpha_k(y)^{p-1}-y^{p-1})$.\label{it:derFour}
\end{enumerate}
\end{lemma}

\begin{proof}
\ref{it:derOne}$\iff$\ref{it:derTwo}:
We wish to show that $\Der_K(A_1^k)=C_{\Der_K(A_1)}(\alpha_k)$. Since $\alpha_k$ is injective, this holds if and only if $\delta(1_{A_1})=0$ for all $\delta\in\Der_K(A_1^k)$ by \autoref{lem:derivations}. Now, we have that $\delta(1_{A_1})=\delta(\alpha_k(1_{A_1}))=\delta(1_{A_1}*1_{A_1})=\delta(1_{A_1})*1_{A_1}+1_{A_1}*\delta(1_{A_1})=2\alpha_k(\delta(1_{A_1})$. Hence, if $p=2$, then $\delta(1_{A_1})=0$. Therefore, assume $p\neq2$ and put $\delta(1_{A_1})=\sum_{i\in\mathbb{N}}\sum_{j\in\mathbb{N}}d_{ij}y^ix^j$ for $d_{ij}\in K$. The previous equation then reads $\sum_{i\in\mathbb{N}}\sum_{j\in\mathbb{N}}d_{ij}y^ix^j=\sum_{i\in\mathbb{N}}\sum_{j\in\mathbb{N}}2d_{ij}\alpha_k(y)^i x^j$, so $\sum_{i\in\mathbb{N}}d_{ij}y^i=\sum_{i\in\mathbb{N}}2d_{ij}\alpha_k(y)^i$ for all $j\in\mathbb{N}$. Now, $\alpha_k(y)=k_0+y+k_py^p+k_{2p}y^{2p}+\cdots$, so by comparing the degrees of the two sums, $k_p=k_{2p}=k_{3p}=\ldots=0$ unless all $d_{ij}$ are zero. If $k_p=k_{2p}=k_{3p}=\ldots=0$ we are left with  the equation $\sum_{i\in\mathbb{N}}d_{ij}y^i=\sum_{i\in\mathbb{N}}2d_{ij}(k_0+y)^i$, which has no solution unless  $d_{ij}=0$ for all $i,j\in\mathbb{N}$.  So $\delta(1_{A_1})=0$, and therefore $\Der_K(A_1^k)=C_{\Der_K(A_1)}(\alpha_k)$.

\ref{it:derTwo}$\iff$\ref{it:derThree}: Since both $\alpha_k$ and $\delta$ are linear, \ref{it:derTwo} is equivalent to $\alpha_k(\delta(y^mx^n))=\delta(\alpha_k(y^mx^n))$ for $m,n\in\mathbb{N}$. Now, we claim that this in turn is equivalent to $\alpha_k(\delta(x))=\delta(\alpha_k(x))$ and $\alpha_k(\delta(y))=\delta(\alpha_k(y))$. Clearly the former condition implies the latter condition. We wish to show, by induction over $m$ and $n$, that also the latter condition implies the former condition. To this end, assume that $\alpha_k(\delta(x))=\delta(\alpha_k(x))$ and $\alpha_k(\delta(y))=\delta(\alpha_k(y))$. First, $\alpha_k(\delta(1_{A_1}))=\alpha_k(0)=0$ and $\delta(\alpha_k(1_{A_1}))=\delta(1_{A_1})=0$, so the base case $m=n=0$ holds. Now, let $n$ be fixed, and assume that the induction hypothesis holds. Then,
\begin{align*}
&\alpha_k(\delta(y^{m+1}x^n))=\alpha_k(\delta(y)\cdot y^mx^n+y\cdot\delta(y^mx^n))\\
&=\alpha_k(\delta(y))\cdot\alpha_k(y^mx^n)+\alpha_k(y)\cdot\alpha_k(\delta(y^mx^n))\\
&=\delta(\alpha_k(y))\cdot\alpha_k(y^mx^n)+\alpha_k(y)\cdot\delta(\alpha_k(y^mx^n))\\
&=\delta(\alpha_k(y)\cdot\alpha_k(y^mx^n))=\delta(\alpha_k(y^{m+1}x^n)).
\end{align*}
Similarly, $\alpha_k(\delta(y^mx^{n+1}))=\delta(\alpha_k(y^mx^{n+1}))$ whenever $m$ is fixed. By using that $\delta(\alpha_k(x))=\delta(x)$ and $\delta(\alpha_k(y))=k_0\delta(1_{A_1})+\delta(y)+k_p\delta(y^p)+k_{2p}\delta(y^{2p})+\cdots=\delta(y)+k_p\delta(y^p)+k_{2p}\delta(y^{2p})+\cdots$, the end result follows. 

\ref{it:derThree}$\implies$\ref{it:derFour}: Let $\delta\in\Der_K(A_1)$, and assume $\alpha_k(\delta(x))=\delta(x)$ and $\alpha_k(\delta(y))=\delta(y)+k_p\delta(y^p)+k_{2p}\delta(y^{2p})+\cdots$. By the proof of the preceding equivalence, we then have $\delta\in C_{\Der_K(A_1)}(\alpha_k)$, so $\alpha_k(\delta(x^p))=\delta(\alpha_k(x^p))$ and $\alpha_k(\delta(y^p))=\delta(\alpha_k(y^p))$. Moreover, by \autoref{thm:BLO15}, $\delta\in\Der_K(A_1)\implies\delta=uE_x+vE_y+\ad_q$, where $E_x,E_y\in \Der_K(A_1)$ are defined by $E_x(x)=y^{p-1}$, $E_x(y)=E_y(x)=0$, $E_y(y)=x^{p-1}$, and $u,v\in K[x^p,y^p]$ and $q\in A_1$. By Lemma 3.6 in \cite{BLO15}, $E_x=-\frac{\mathrm{d}}{\mathrm{d}(x^p)}$ on $K[x^p]$ and $E_y=-\frac{\mathrm{d}}{\mathrm{d}(y^p)}$ on $K[y^p]$, so by using that $x^p,y^p\in K[x^p,y^p]=C(A_1)$, we have $\alpha_k(\delta(x^p))=\alpha_k(-u)$, $\delta(\alpha_k(x^p))=-u$, $\alpha_k(\delta(y^p))=\alpha_k(-v)$, and $\delta(\alpha_k(v))=\delta(-v)$. Hence we must have $\alpha_k(u)=u$ and $\alpha_k(v)=v$. Now, put $q=\sum_{i\in\mathbb{N}}\sum_{j\in\mathbb{N}}q_{ij}y^ix^j$ for some $q_{ij}\in K$. Then $\alpha_k\left(\frac{\partial q}{\partial y}\right)=\sum_{i\in\mathbb{N}}\sum_{j\in\mathbb{N}}q_{ij}i\alpha_k(y)^{i-1}x^j=\sum_{i\in\mathbb{N}}\sum_{j\in\mathbb{N}}q_{ij}i(k_0+y+k_py^p+k_{2p}y^{2p}+\cdots)^{i-1}x^j=\frac{\partial \alpha_k(q)}{\partial y}$ and $\alpha_k\left(\frac{\partial q}{\partial x}\right)=\sum_{i\in\mathbb{N}}\sum_{j\in\mathbb{N}}q_{ij}j\alpha_k(y)^ix^{j-1}=\frac{\partial \alpha_k(q)}{\partial x}$. Using this and \autoref{lem:commutator-derivative}, $\alpha_k(\delta(x))=\alpha_k(u\cdot y^{p-1}+[q,x])=u\cdot\alpha_k(y)^{p-1}-\alpha_k\left(\frac{\partial q}{\partial y}\right)=u\cdot\alpha_k(y)^{p-1}-\frac{\partial \alpha_k(q)}{\partial y}$, while $\delta(x)=u\cdot y^{p-1}-\frac{\partial q}{\partial y}$. We also have $\alpha_k(\delta(y))=\alpha_k(v\cdot x^{p-1}+[q,y])=v\cdot x^{p-1}+\alpha_k\left(\frac{\partial q}{\partial x}\right)=v\cdot x^{p-1}+\frac{\partial\alpha_k(q)}{\partial x}$, and since $y^p,y^{2p},y^{3p},\ldots\in K[y^p]\subset C(A_1)$, we have $\delta(y)+k_p\delta(y^p)+k_{2p}\delta(y^{2p})+\cdots=\delta(\alpha_k(y))=v\cdot x^{p-1}-v\cdot(k_p+2k_{2p}y^{p}+3k_{3p}y^{2p}+\cdots)+\frac{\partial q}{\partial x}=v\cdot x^{p-1}-v\cdot\frac{\mathrm{d}}{\mathrm{d}(y^p)}(\alpha_k(y)-y)+\frac{\partial q}{\partial x}$. Hence, we can conclude that $\frac{\partial}{\partial x}\left(\alpha_k(q)-q\right)=v\cdot\frac{\mathrm{d}}{\mathrm{d}(y^p)}(y-\alpha_k(y))$ and $\frac{\partial}{\partial y}\left(\alpha_k(q)-q\right)=u\cdot(\alpha_k(y)^{p-1}-y^{p-1})$.

\ref{it:derFour}$\implies$\ref{it:derThree}: Assume that \ref{it:derFour} holds. Then, by \autoref{thm:BLO15}, $\delta\in\Der_K(A_1)$. Moreover, by the very same calculation as in the proof of the preceding implication, $\alpha_k(\delta(x))=\delta(x)$ and $\alpha_k(\delta(y))=\delta(y)+k_p\delta(y^p)+k_{2p}\delta(y^{2p})+\cdots$. 
\end{proof}

When $k=0$, all the derivations of $A_1^k$ are described in \autoref{thm:BLO15}. The next two propositions deal with the case when $k\neq0$.

\begin{proposition}\label{prop:der-case-1}
If $k=(k_0,0,0,0,\ldots)\neq0$, then $\delta\in\Der_K(A_1^k)$ if and only if $\delta=uE_x+vE_y+\ad_q$. Here $E_x,E_y\in\Der_K(A_1)$ are defined by $E_x(x)=y^{p-1}$, $E_x(y)=E_y(x)=0$, $E_y(y)=x^{p-1}$, $u=\sum_{i\in\mathbb{N}}\sum_{j\in\mathbb{N}}u_{ij}y^ix^j\in K[x^p,y^p]$, $v=\sum_{i\in\mathbb{N}}\sum_{j\in\mathbb{N}}v_{ij}y^ix^j\in K[x^p,y^p]$, and $q=\sum_{i\in\mathbb{N}}\sum_{j\in\mathbb{N}}q_{ij}y^{i}x^{j}\in A_1$ for some $q_{ij},u_{ij},v_{ij}\in K$, satisfying, for all $j,l\in\mathbb{N}$ and $m\in\mathbb{N}_{>0}$,
\begin{align*}
&\sum_{l+1\leq i}\binom{i}{l}v_{ij}k_0^{i}=\sum_{l+1\leq i}\binom{i}{l}u_{ij}k_0^{i}=\sum_{l+1\leq i}\binom{i}{l}jq_{ij}k_0^{i}=0,\\
&\sum_{m+1\leq i}\binom{i}{m}mq_{ij}k_0^{i}=	\sum_{i=0}^{m-1}\binom{p-1}{m-i} u_{ij}k_0^{i+p-1}.
\end{align*}
\end{proposition}

\begin{proof}
Let $k=(k_0,0,0,0,\ldots)\neq0$. By \autoref{lem:derivation-unit}, $\delta\in\Der_K(A_1^k)\iff \delta=uE_x+vE_y+\ad_q$, where $E_x,E_y\in \Der_K(A_1)$ are defined by $E_x(x)=y^{p-1}$, $E_x(y)=E_y(x)=0$, $E_y(y)=x^{p-1}$, and $u,v\in K[x^p,y^p]$ and $q\in A_1$ satisfy $\alpha_k(u)=u$, $\alpha_k(v)=v$, $\frac{\partial}{\partial x}\left(\alpha_k(q)-q\right)=0$, and $\frac{\partial}{\partial y}\left(\alpha_k(q)-q\right)=u\cdot((k_0+y)^{p-1}-y^{p-1})$. We wish to rewrite these last four conditions. To this end, let $u=\sum_{i\in\mathbb{N}}\sum_{j\in\mathbb{N}}u_{ij}y^ix^j$ for some $u_{ij}\in K$. Then, by the binomial theorem, $\alpha_k(u)=\sum_{i\in\mathbb{N}}\sum_{j\in\mathbb{N}}u_{ij}(k_0+y)^i x^j=\sum_{i\in\mathbb{N}}\sum_{j\in\mathbb{N}}\sum_{l=0}^{i}\binom{i}{l}u_{ij}k_0^{i-l}y^lx^j$. Hence, we have that $\alpha_k(u)=u\iff\sum_{i\in\mathbb{N}_{>0}}\sum_{j\in\mathbb{N}}\sum_{l=0}^{i-1} \binom{i}{l}u_{ij}k_0^{i-l}y^lx^j=0\iff \sum_{i\in\mathbb{N}_{>0}}\sum_{l=0}^{i-1} \binom{i}{l}u_{ij}k_0^{i-l}y^l=0$ for all $j\in\mathbb{N}\iff \sum_{l\in\mathbb{N}}\sum_{l+1\leq i}\binom{i}{l}u_{ij}k_0^{i-l}y^l=0$ for all $j\in\mathbb{N}\iff\sum_{l+1\leq i}\binom{i}{l}u_{ij}k_0^{i}=0$ for all $j,l\in\mathbb{N}$. Similarly, if $v=\sum_{i\in\mathbb{N}}\sum_{j\in\mathbb{N}}v_{ij}y^ix^j$ for some $v_{ij}\in K$, then $\alpha_k(v)=v\iff \sum_{l+1\leq i}\binom{i}{l}v_{ij}k_0^{i}=0$ for all $j,l\in\mathbb{N}$. Now, put $q=\sum_{i\in\mathbb{N}}\sum_{j\in\mathbb{N}}q_{ij}y^{i}x^{j}$ for some $q_{ij}\in K$. Then, $\alpha_k(q)-q=\sum_{i\in\mathbb{N}_{>0}}\sum_{j\in\mathbb{N}}\sum_{l=0}^{i-1}\binom{i}{l}q_{ij}k_0^{i-l}y^{l}x^{j}$, so by defining $0x^{-1}:=0$, $\frac{\partial}{\partial x}(\alpha_k(q)-q)=0\iff \sum_{i\in\mathbb{N}_{>0}}\sum_{j\in\mathbb{N}}\sum_{l=0}^{i-1}\binom{i}{l}jq_{ij}k_0^{i-l}y^{l}x^{j-1}=0\iff \sum_{l+1\leq i}\binom{i}{l}jq_{ij}k_0^{i}=0$ for all $j,l\in\mathbb{N}$. We have that $\frac{\partial}{\partial y}(\alpha_k(q)-q)=\sum_{i\in\mathbb{N}_{>0}}\sum_{j\in\mathbb{N}}\sum_{l=1}^{i-1}\binom{i}{l}lq_{ij}k_0^{i-l}y^{l-1}x^{j}$, and since $(k_0+y)^{p-1}=\sum_{n=0}^{p-1}\binom{p-1}{n} k_0^{p-1-n}y^{n}$, we get $u\cdot((k_0+y)^{p-1}-y^{p-1})=u\cdot\sum_{n=0}^{p-2}\binom{p-1}{n} k_0^{p-1-n}y^{n}$. Using that $u\in K[x^p,y^p]=C(A_1)$, we can rewrite the last condition as $\sum_{i\in\mathbb{N}_{>0}}\sum_{j\in\mathbb{N}}\sum_{l=1}^{i-1}\binom{i}{l}lq_{ij}k_0^{i-l}y^{l-1}x^{j}=\sum_{i\in\mathbb{N}}\sum_{j\in\mathbb{N}}\sum_{n=0}^{p-2}\binom{p-1}{n}u_{ij}k_0^{p-1-n}y^{i+n}x^j$. Continuing, this is in turn equivalent to $\sum_{i\in\mathbb{N}_{>0}}\sum_{l=1}^{i-1}\binom{i}{l}lq_{ij}k_0^{i-l}y^{l-1}=\sum_{i\in\mathbb{N}}\sum_{n=0}^{p-2}\binom{p-1}{n}u_{ij}k_0^{p-1-n}y^{i+n}$ for all $j\in\mathbb{N}$. Now, put $m:=i+n+1$ in the latter double sum, so that the above equality reads $\sum_{i\in\mathbb{N}_{>0}}\sum_{l=1}^{i-1}\binom{i}{l}lq_{ij}k_0^{i-l}y^{l-1}=\sum_{i\in\mathbb{N}}\sum_{m=i+1}^{i+p-1}\binom{p-1}{m-i}u_{ij}k_0^{i+p-m-1}y^{m-1}$ for all $j\in\mathbb{N}$. Last, we define $\binom{a}{b}:=0$ whenever $a<b$, and then change the order of summation in both double sums, resulting in $\sum_{l\in\mathbb{N}_{>0}}\sum_{l+1\leq i}\binom{i}{l}lq_{ij}k_0^{i-l}y^{l-1}=\sum_{m\in\mathbb{N}_{>0}}\sum_{i=0}^{m-1}\binom{p-1}{m-i}u_{ij}k_0^{i+p-m-1}y^{m-1}$ for all $j\in\mathbb{N}$. By comparing coefficients, $\sum_{m+1\leq i}\binom{i}{m}mq_{ij}k_0^{i}=\sum_{i=0}^{m-1}\binom{p-1}{m-i}u_{ij}k_0^{i+p-1}$ for all $j\in\mathbb{N}$ and $m\in\mathbb{N}_{>0}$.
 \end{proof}

\begin{proposition}\label{prop:der-case-2}
If $k=(k_0,k_p,k_{2p},\ldots,k_{Mp},0,0,0,\ldots)$ for some $M\in\mathbb{N}_{>0}$ where $k_{Mp}\neq0$, then $\delta\in\Der_K(A_1^k)$ if and only if 
\begin{equation*}
	\delta=\begin{cases}
		vE_y+\ad_q &\text{if } k=(k_0,0,\ldots, 0,k_{p^2},0,\ldots,0,k_{2p^2},0,\ldots),\\
		\ad_q &\text{else}. 
	\end{cases}
\end{equation*}
Here, $E_y\in\Der_K(A_1)$ is defined by $E_y(x)=0$, $E_y(y)=x^{p-1}$, $v\in K[x^p]$ and $q=ayx+\sum_{i\equiv0\pmod{p}}b_{i}y^ix+c_iyx^i$ for some $a,b_i,c_i\in K$.
\end{proposition}

\begin{proof}
Let $k=(k_0,k_p,k_{2p},\ldots,k_{Mp},0,0,0,\ldots)$ for some $M\in\mathbb{N}_{>0}$ where $k_{Mp}\neq0$. By \autoref{lem:derivation-unit}, $\delta\in\Der_K(A_1^k)\iff \delta=uE_x+vE_y+\ad_q$, where $E_x,E_y\in \Der_K(A_1)$ are defined by $E_x(x)=y^{p-1}$, $E_x(y)=E_y(x)=0$, $E_y(y)=x^{p-1}$, and $u,v\in K[x^p,y^p]$ and $q\in A_1$ satisfy $\alpha_k(u)=u$, $\alpha_k(v)=v$, $\frac{\partial}{\partial x}\left(\alpha_k(q)-q\right)=-v\cdot(k_p+2k_{2p}y^p+3k_{3p}y^{3p}+\cdots+Mk_{Mp}y^{(M-1)p})$, and $\frac{\partial}{\partial y}\left(\alpha_k(q)-q\right)=u\cdot((k_0+y+k_py^p+k_{2p}y^{2p}+\cdots+k_{Mp}y^{Mp})^{p-1}-y^{p-1})$. We wish to rewrite these last four conditions. Since  $k_{Mp}\neq0$ where $M\in\mathbb{N}_{>0}$, $\alpha_k(u)=u\iff u\in K[x^p]$ and $\alpha_k(v)=v\iff v\in K[x^p]$. Continuing, we put $q=\sum_{i,j\in\mathbb{N}}q_{ij}y^ix^j$ for some $q_{ij}\in K$ and examine the last condition. Since $u\in K[x^p]\subset C(A_1)$, by comparing the coefficients for $y^{p-1}$, we see that $u=0$, which in turn gives $q=\sum_{i\equiv0\pmod{p}}\sum_{j\in\mathbb{N}}(q_{1j}y+q_{ij}y^i) x^j$. Now, $\alpha_k$ and $\frac{\partial}{\partial x}$ commute (see the second last part of the proof of \autoref{lem:derivation-unit}), so the third condition can be written as $\alpha_k\left(\frac{\partial q}{\partial x}\right)=\frac{\partial q}{\partial x}-v\cdot\left(k_p+2k_{2p}y^p+3k_{3p}y^{2p}+\cdots+Mk_{Mp}y^{(M-1)p}\right)$. Assume $k_p=2k_{2p}=3k_{3p}=\ldots=Mk_{Mp}=0$, which is equivalent to $k=(k_0,0,\ldots, 0,k_{p^2},0,\ldots,0,k_{2p^2},0,\ldots)$, or that $v=0$. The previous equality then reads $\alpha_k\left(\frac{\partial q}{\partial x}\right)=\frac{\partial q}{\partial x}$, which is equivalent to $\frac{\partial q}{\partial x}\in K[x]$, so $q=\sum_{i,j\equiv0\pmod{p}}(q_{11}y+q_{i1}y^i) x+(q_{1j}y+q_{ij}y^i) x^j$. Now, assume instead that $v\neq0$ and that not all of $k_p, 2k_{2p},3k_{3p}, \ldots, Mk_{Mp}$ are zero. Let $L$ be such that $Lk_{Lp} =0$ and $\ell k_{\ell p} =0$ for $\ell >L$. Then, with $q_{ij}$ possibly nonzero only if $i=1$ or $i\equiv0\pmod{p}$,  
\begin{align*}
&&&\alpha_k\left(\frac{\partial q}{\partial x}\right)=\frac{\partial q}{\partial x}-v\cdot\left(k_p+2k_{2p}y^p+3k_{3p}y^{2p}+\cdots+Mk_{Mp}y^{(M-1)p}\right)\\
&&\iff &\sum_{i,j\in\mathbb{N}}jq_{ij}\alpha_k(y)^i x^{j-1}\\
&&&=-v\cdot\left(k_p+2k_{2p}y^p+3k_{3p}y^{2p}+\cdots+Lk_{Lp}y^{(L-1)p}\right)+\sum_{i,j\in\mathbb{N}}jq_{ij}y^ix^{j-1}\\
&&\iff &\sum_{i,j\in\mathbb{N}}jq_{ij}\left(k_0+y+k_py^p+k_{2p}y^{2p}+\cdots+k_{Mp}y^{Mp}\right)^i x^{j-1}\\
&&&=-v\cdot\left(k_p+2k_{2p}y^p+3k_{3p}y^{2p}+\cdots+Lk_{Lp}y^{(L-1)p}\right)+\sum_{i,j\in\mathbb{N}}jq_{ij}y^ix^{j-1}
\end{align*}
By comparing the degrees of the left-hand side and the right-hand side, we realize that the above equation has no solution. Hence, if we put  $a:=q_{11}, b_i:=q_{i1}, c_j:=q_{1j}$ and note that $q_{ij}y^ix^j\in C(A_1)$ for $i,j\equiv0\pmod{p}$, the result follows.
\end{proof}

\begin{lemma}\label{lem:homo-necessary-sufficient}
The following statements are equivalent:
\begin{enumerate}[label=(\roman*)]
	\item $f\in\Hom_K(A_1^k,A_1^l)$ and $f\neq 0$.\label{it:homOne}
	\item $f\in C_{\End_K(A_1)}(\alpha_k,\alpha_l)$ and $f\neq 0$.\label{it:homTwo}
	\item $f\in\End_K(A_1)$ satisfying $f(x)=\alpha_l(f(x))$ and $k_0+f(y)+k_pf(y)^p+k_{2p}f(y)^{2p}+\cdots=\alpha_l(f(y))$.\label{it:homThree}
\end{enumerate}
\end{lemma}

\begin{proof}\ref{it:homOne}$\iff$\ref{it:homTwo}: By \autoref{lem:alpha-injective}, $\alpha_l$ is injective, and so the statement follows from \autoref{lem:alpha-beta-morphism-commute}.

\ref{it:homTwo}$\iff$\ref{it:homThree}: Using that $\alpha_k(x)=x$, $\alpha_k(y)=k_0+y+k_py^p+k_{2p}y^{2p}+\cdots$, and $f(1_{A_1})=1_{A_1}$ (see \autoref{subsec:weyl-algebra}), it is immediate that \ref{it:homTwo} implies \ref{it:homThree}. We wish to show that also \ref{it:homThree} implies \ref{it:homTwo}. To this end, assume that \ref{it:homThree} holds, i.e. that $f(\alpha_k(x))=\alpha_l(f(x))$ and $f(\alpha_k(y))=\alpha_l(f(y))$. Now, since $\alpha_k,\alpha_l$, and $f$ are linear, it is sufficient to show that $f(\alpha_k(y^mx^n))=\alpha_l(f(y^mx^n))$ for any $m,n\in\mathbb{N}$. Since $f(\alpha_k(1_{A_1}))=f(1_{A_1})=1_{A_1}$ and $\alpha_l(f(1_{A_1}))=\alpha_l(1_{A_1})=1_{A_1}$, we have that
\begin{align*}
&f(\alpha_k(y^mx^n))=f(\alpha_k(y)^m\cdot\alpha_k(x)^n)=f(\alpha_k(y))^m\cdot f(\alpha_k(x))^n\\
&=\alpha_l(f(y))^m\cdot\alpha_l(f(x))^n=\alpha_l(f(y)^m)\cdot\alpha_l(f(x)^n)=\alpha_l(f(y^m))\cdot\alpha_l(f(x^n))\\
&=\alpha_l(f(y^m)\cdot f(x^n))=\alpha_l(f(y^mx^n)). 
\end{align*}
\end{proof}

\begin{proposition}\label{prop:hom-injective}
Every nonzero endomorphism on $A_1^k$ is injective.	
\end{proposition}

\begin{proof}
With $k=l$ in \autoref{lem:homo-necessary-sufficient}, every  endomorphism on $A_1^k$ is an endomorphism on $A_1$, and by Lemma 17 in \cite{Tsu03}, every nonzero endomorphism on $A_1$ is injective.
\end{proof}

In characteristic zero, if $k\neq0$, then every nonzero endomorphism on $A_1^k$ is an automorphism (cf. Corollary 5 in \cite{BR20}). If $k=0$, then any nonzero endomorphism is injective since $A_1$ is simple. In \cite{Dix68}, Dixmier asked if all nonzero endomorphisms on $A_1$ are also surjective? The question still remains unanswered, and the statement that all 	nonzero endomorphisms are surjective is now known as the Dixmier conjecture. Tsuchimoto \cite{Tsu05} and Kanel-Belov and Kontsevich \cite{KBK07} have proved, independently, that the Dixmier conjecture is stably equivalent to the more famous Jacobian conjecture. In prime characteristic, it is known that not every nonzero endomorphism on $A_1$ is surjective, however. The next proposition demonstrates that this fact generalizes to $A_1^k$ for arbitrary $k$.

\begin{proposition}\label{prop:hom-non-Dixmier}
Not every nonzero endomorphism on $A_1^k$ is surjective.	
\end{proposition}

\begin{proof}
Since $A_1$ is a free algebra on the letters $x$ and $y$ modulo the commutation relation $x\cdot y-y\cdot x=1_{A_1}$, we may define an endomorphism on $A_1$ by defining it arbitrary on $x$ and $y$ and then extending it linearly and multiplicatively, as long as it respects the above commutation relation. Since $x^p\in C(A_1)$, $f$ defined by $f(x)=x+x^p$ and $f(y)=y$ satisfies $f(x\cdot y-y\cdot x)=f(1_{A_1})$ and hence defines an endomorphism on $A_1$. However, $f$ is not surjective. Assume the contrary, and let $x=f(q)$ for some $q=\sum_{i\in\mathbb{N}}q_i(y)x^i$, where $q_i(y)\in K[y]$. Then $1=\deg(x)=\deg(f(q))=\deg(q)\deg(f(x))=\deg(q)p$, which is a contradiction. We claim that $f\in\End_K(A_1^k)$. We have that $f(x)=\alpha_k(f(x))$ and $k_0+f(y)+k_pf(y)^p+k_{2p}f(y)^{2p}+\cdots=\alpha_k(f(y))$, so with $k=l$ in \autoref{lem:homo-necessary-sufficient}, $f\in\End_K(A_1^k)$.
\end{proof}

It is clear, for instance from \autoref{prop:power-assoc}, that $A_1^k$ is associative if and only if $k=0$, and hence $A_1^0\cong A_1^l$ if and only if $l=0$. The next two propositions deal with the case when $k\neq0$.

\begin{proposition}\label{prop:k-zero-isomorphism}
If $k=(k_0,0,0,0,\ldots)\neq0$, then $A_1^k\cong A_1^l$ if and only if $l=(l_0,0,0,0,\ldots)\neq0$.	
\end{proposition}

\begin{proof}
Let $k=(k_0,0,0,0,\ldots)\neq0$ and assume $A_1^k\cong A_1^l$. Since $A_1^k$ is not associative, $A_1^l$ cannot be associative, so $l\neq0$. Put $l=(l_0,l_p,l_{2p},\ldots,l_{Np},0,0,0)$ for some $N\in\mathbb{N}$, where $l_{Np}\neq 0$. We now wish to show that there is no nonzero homomorphism, let alone isomorphism, from $A_1^k$ to $A_1^l$ unless $N=0$. To this end, assume $f\colon A_1^k\to A_1^l$ is a nonzero homomorphism. By \autoref{lem:homo-necessary-sufficient}, we must have $k_0+f(y)=\alpha_l(f(y))$. Now, put $f(y)=\sum_{i=0}^m\sum_{j=0}^na_{ij}y^ix^j$ for some $m,n\in \mathbb{N}$ and $a_{ij}\in K$. Then, $\alpha_l(f(y))=\sum_{i=0}^m\sum_{j=0}^na_{ij}(l_0+y+l_py^p+l_{2p}y^{2p}+\cdots+l_{Np}y^{Np})^ix^j$. Hence $k_0+f(y)=\alpha_l(f(y))$ if and only if 
$k_0+\sum_{i=0}^m\sum_{j=0}^na_{ij}y^ix^j = \sum_{i=0}^m\sum_{j=0}^na_{ij}(l_0+y + l_py^p+l_{2p}y^{2p}+\cdots+l_{Np}y^{Np})^ix^j$. We note that if $m=0$, then $k_0+\sum_{j=0}^na_{0j}x^j=\sum_{j=0}^na_{0j}x^j\iff k_0=0$, which is a contradiction. Hence $m\neq0$, so by comparing degrees in $y$, we must have $N=0$. 

Now, assume that $k=(k_0,0,0,0,\ldots)\neq0$ and $l=(l_0,0,0,0,\ldots)\neq0$. Define a homomorphism $g$ on $A_1$ by $g(x):=\frac{l_0}{k_0}x$ and $g(y):=\frac{k_0}{l_0}y$. By \autoref{thm:Mak84}, $g$ is then not only a homomorphism, but also a linear automorphism. Moreover, $g(x)=\alpha_l(g(x))$ and $k_0+g(y)=\alpha_l(g(y))$, so by \autoref{lem:homo-necessary-sufficient},  $g\colon A_1^k\to A_1^l$ is an isomorphism.
\end{proof}

\begin{proposition}\label{prop:iso-necessary-sufficient}Let $k=(k_0,k_p,k_{2p},\ldots,k_{Mp},0,0,0,\ldots)$ for some $M\in\mathbb{N}_{>0}$ where $k_{Mp}\neq0$ and let $l=(l_0,l_p,l_{2p},\ldots,l_{Np},0,0,0,\ldots)$ for some $N\in\mathbb{N}_{>0}$ where $l_{Np}\neq0$. Then a map $f\colon A_1^k\to A_1^l$ is an isomorphism if and only if $M=N$, $f\in\Aut_K(A_1)$ of the form $f(x)=b_0+a_1^{-1}x$ and $f(y)=a_0+a_1y$ for $a_0,b_0\in K$, $a_1\in K^\times$, satisfying 
\begin{equation}
\sum_{i=j}^M\binom{i}{j}k_{ip}a_0^{(i-j)p}a_1^{jp-1}=l_{jp},\quad 0\leq j\leq M.\label{eq:isomorphism-equation}
\end{equation}
\end{proposition}

\begin{proof}
Let $k=(k_0,k_p,k_{2p},\ldots,k_{Mp},0,0,0,\ldots)$, $l=(l_0,l_p,l_{2p},\ldots,l_{Np},0,0,0,\ldots)$ for some $M,N\in\mathbb{N}_{>0}$ where $k_{Mp},l_{Np}\neq0$, and assume that $f\colon A_1^k\to A_1^l$ is a map. By \autoref{lem:homo-necessary-sufficient}, $f$ is a nonzero homomorphism if and only if $f$ is an endomorphism on $A_1$ satisfying $f(x)=\alpha_l(f(x))$ and $k_0+f(y)+k_pf(y)^p+k_{2p}f(y)^{2p}+\cdots+k_{Mp}f(y)^{Mp}=\alpha_l(f(y))$. Since $l_{Np}\neq0$ where $N\in\mathbb{N}_{>0}$, from the definition of $\alpha_l$, $f(x)=\alpha_l(f(x))\iff f(x)\in K[x]$. Let us put $f(y)=\sum_{i=0}^m\sum_{j=0}^n a_{ij}y^ix^j$ for some $m,n\in\mathbb{N}$ and $a_{ij}\in K$ where $a_{mn}\neq0$. Then, 
\begin{align*}
&k_0+f(y)+k_pf(y)^p+k_{2p}f(y)^{2p}+\cdots+k_{Mp}f(y)^{Mp}\\
&=k_0+\sum_{i=0}^m\sum_{j=0}^na_{ij}y^ix^j+k_p\left(\sum_{i=0}^m\sum_{j=0}^na_{ij}y^ix^j\right)^p+k_{2p}\left(\sum_{i=0}^m\sum_{j=0}^na_{ij}y^ix^j\right)^{2p}\\
&+\cdots+k_{Mp}\left(\sum_{i=0}^m\sum_{j=0}^na_{ij}y^ix^j\right)^{Mp},\\
&\alpha_l(f(y))=\sum_{i=0}^m\sum_{j=0}^na_{ij}\alpha_l(y)^ix^j\\
&=\sum_{i=0}^m\sum_{j=0}^na_{ij}(l_0+y+l_py^p+l_{2p}y^{2p}+\cdots+l_{Np}y^{Np})^ix^j\\
&=\sum_{i=0}^m\sum_{j=0}^n\sum_{\ell=0}^i\binom{i}{\ell}a_{ij}(l_0+l_py^p+l_{2p}y^{2p}+\cdots+l_{Np}y^{Np})^{i-\ell} y^{\ell}x^j.
\end{align*}
Hence $k_0+f(y)+k_pf(y)^p+k_{2p}f(y)^{2p}+\cdots+k_{Mp}f(y)^{Mp}=\alpha_l(f(y))$ if and only if
\begin{align*}
&k_0+k_p\left(\sum_{i=0}^m\sum_{j=0}^na_{ij}y^ix^j\right)^p+k_{2p}\left(\sum_{i=0}^m\sum_{j=0}^na_{ij}y^ix^j\right)^{2p}\\
&+\cdots +k_{Mp}\left(\sum_{i=0}^m\sum_{j=0}^na_{ij}y^ix^j\right)^{Mp}\\
&=\sum_{i=1}^m\sum_{j=0}^n\sum_{\ell=0}^{i-1}\binom{i}{\ell}a_{ij}(l_0+l_py^p+l_{2p}y^{2p}+\cdots+l_{Np}y^{Np})^{i-\ell} y^{\ell}x^j.
\end{align*}
By comparing degrees in $y$, we realize that $M=N$, which in turn gives $n=0$, so 
\begin{align}
&k_0+k_p\left(\sum_{i=0}^ma_{i0}y^i\right)^p+k_{2p}\left(\sum_{i=0}^ma_{i0}y^i\right)^{2p}+\cdots+k_{Mp}\left(\sum_{i=0}^ma_{i0}y^i\right)^{Mp}\nonumber\\
&=\sum_{i=1}^m\sum_{\ell=0}^{i-1}\binom{i}{\ell}a_{i0}(l_0+l_py^p+l_{2p}y^{2p}+\cdots+l_{Mp}y^{Mp})^{i-\ell} y^{\ell}.\label{eq:k-l-equation}
\end{align}
Hence $f\colon A_1^k\to A_1^l$ is a nonzero homomorphism if and only if $M=N$, $f\in\End_K(A_1)$ satisfying $f(x)\in K[x]$, $f(y)\in K[y]$, and \eqref{eq:k-l-equation}. Now, we wish to know when $f$ is an isomorphism. To this end, first assume that $f$ is a surjective homomorphism. Then there is some $r:=\sum_{i=0}^{m'}\sum_{j=0}^{n'}r_{ij}y^ix^j$ with $m',n'\in\mathbb{N}$ and $r_{ij}\in K$, such that $y=f(r)$. Since $y=f(r)=\sum_{i=0}^{m'}\sum_{j=0}^{n'}r_{ij}f(y)^i\cdot f(x)^j$ where $f(x)\in K[x]$ and $f(y)\in K[y]$, by comparing degrees, $n'=0$ and $m'=1$. This in turn gives $m=1$, so that $f(y)=a_{00}+a_{10}y$ where $a_{10}\neq0$. After some reindexing, $\eqref{eq:k-l-equation}$ now reads
\begin{align}
&\sum_{i=0}^M k_{ip}(a_{00}+a_{10}y)^{ip}=a_{10}\sum_{i=0}^Ml_{ip}y^{ip}	\nonumber\\
\iff &\sum_{i=0}^M\sum_{j=0}^i\binom{i}{j}k_{ip}a_{00}^{(i-j)p}a_{10}^{jp}y^{jp}=a_{10}\sum_{i=0}^Ml_{ip}y^{ip}\nonumber\\
\iff &\sum_{i=j}^M\binom{i}{j}k_{ip}a_{00}^{(i-j)p}a_{10}^{jp-1}=l_{jp},\quad 0\leq j\leq M.\label{eq:k-l-equation2}
\end{align}
By a similar argument, $f(x)=b_0+b_1x$ for some $b_0,b_1\in K$ where $b_1\neq0$. From $f(x)\cdot f(y)-f(y)\cdot f(x)=f(1_{A_1})=1_{A_1}$, it follows that $b_1=a_{10}^{-1}$. Now, does this define an isomorphism? Yes, under the assumption that $\eqref{eq:k-l-equation2}$ is satisfied and that $f$ is an automorphism on $A_1$. The latter can for instance be shown by first introducing the following functions, all being automorphisms on $A_1$ by \autoref{thm:Mak84},
\begin{align*}
g_1(x):=&a_{10}^{-1}x, &g_2(x):=&x, &g_3(x):=&y, &g_4(x):=&x,\\
g_1(y):=&a_{10}y, &g_2(y):=&y+a_{00}a_{10}^{-1}, &g_3(y):=&-x, &g_4(y):=&y+b_0a_{10},
\end{align*}
$g_5:=-g_3$, and then noting that $f=g_5\circ g_4\circ g_3\circ g_2\circ g_1$. The result now follows with $a_0:=a_{00}$ and $a_1:=a_{10}$. 
\end{proof}

The two preceding propositions implicitly classify all hom-associative Weyl algebras up to isomorphism. However, it is not obvious under what circumstances there exist (does not exist) constants $a_0,a_1$ that solve \eqref{eq:isomorphism-equation}, hence giving rise to isomorphic (non-isomorphic) hom-associative Weyl algebras. In the next corollary, we study a particular family of hom-associative Weyl algebras, over a finite field. Even in this particular case, we see that there does indeed exist many non-isomorphic hom-associative Weyl algebras, as opposed to the characteristic zero case.  

\begin{corollary}
If $k=(k_0,0,\ldots,0,k_{p^n},0,\ldots,0,k_{p^{2n}},0,\ldots,0,k_{p^{mn}},0,0,0,\ldots)$ where $k_{p^{mn}}\neq0$ for some $m,n\in\mathbb{N}_{>0}$ and $K=\mathbb{F}_{p^n}$, then $A_1^k\cong A_1^l$ if and only if $k_{jp}=l_{jp}$ for all $j\in\mathbb{N}_{>0}$.
\end{corollary}

\begin{proof}
Let $k$ be as above, and assume $K=\mathbb{F}_{p^n}$ for some $n\in\mathbb{N}_{>0}$. Assume that $A_1^k\cong A_1^l$. By \autoref{prop:iso-necessary-sufficient}, $k_{jp}=l_{jp}=0$ whenever $p^{mn-1}<j$. Now, by \autoref{cor:lucas}, \eqref{eq:isomorphism-equation} in \autoref{prop:iso-necessary-sufficient} reads 
\begin{align}
&k_{jp}a_1^{jp-1}=l_{jp},\quad 1\leq j\leq p^{mn-1},\label{eq:fpq-nonzero}\\
&k_0+k_{p^n}a_0^{p^n}+k_{p^{2n}}a_0^{p^{2n}}+\cdots+k_{p^{mn}}a_0^{p^{mn}}=	l_0a_1.\label{eq:fpq-zero}
\end{align}
By Fermat's little theorem for finite fields, $a_0^{p^n}=a_0$ and $a_1^{p^n-1}=1_{\mathbb{F}_{p^n}}$, and so by induction $a_0^{p^n}=a_0^{p^{2n}}=a_0^{p^{3n}}=\ldots=a_0^{p^{mn}}=a_0$ and $a_1^{p^n-1}=a_1^{p^{2n}-1}=a_1^{p^{3n}-1}=\ldots=a_1^{p^{mn}-1}=1_{\mathbb{F}_{p^n}}$. Hence \eqref{eq:fpq-nonzero} is equivalent to $k_{jp}=l_{jp}$, $1\leq j\leq p^{mn-1}$, and \eqref{eq:fpq-zero} to $k_0+k_{p^n}a_0+k_{p^{2n}}a_0+\cdots+k_{p^{mn}}a_0=l_0a_1$. The two equations thus have a solution $a_0=\frac{l_0a_1-k_0}{k_{p^n}+k_{p^{2n}}+k_{p^{3n}}+\cdots+k_{p^{mn}}}$ and $a_1\in\mathbb{F}_{p^n}^\times$.
\end{proof}

\section{Multi-parameter formal deformations}\label{sec:formal-deformation}
In \cite{MS10a}, Makhlouf and Silvestrov introduced the notions of \emph{one-parameter formal hom-associative deformations} and \emph{one-parameter formal hom-Lie deformations}. In \cite{BR20}, the authors showed that the first hom-associative Weyl algebras in characteristic zero are a one-parameter formal deformation of the first associative Weyl algebra, and that this deformation induces a one-parameter formal hom-Lie deformation of the corresponding Lie algebra, when using the commutator as bracket. In this section, we generalize the two notions above and introduce \emph{multi-parameter formal hom-associative deformations} and \emph{multi-parameter formal hom-Lie deformations} (in \cite{Bac20}, a similar notion for ternary hom-Nambu-Lie algebras was introduced, together with examples thereof). We then show that the first hom-associative Weyl algebras over a field of prime characteristic are a multi-parameter formal hom-associative deformation of the first hom-associative Weyl algebra in prime characteristic, inducing a multi-parameter formal hom-Lie deformation of the corresponding Lie algebra, when using the commutator as bracket. 

Now, let $R$ be an associative, commutative, and unital ring, and $M$ an $R$-module. We denote by $R\llbracket t_1,\ldots,t_n\rrbracket$ the formal power series ring in the indeterminates $t_1,\ldots,t_n$ for some $n\in\mathbb{N}_{>0}$, and by $M\llbracket t_1,\ldots,t_n\rrbracket$ the $R\llbracket t_1,\ldots,t_n\rrbracket$-module of formal power series in the same indeterminates, but with coefficients in $M$. This allows us to define a hom-associative algebra $(M\llbracket t_1,\ldots,t_n\rrbracket,\cdot_t,\alpha_t)$ over $R\llbracket t_1,\ldots,t_n\rrbracket$ where $t:=(t_1,\ldots,t_n)$. With some abuse of notation, we say that we extend a map $f\colon M\to M$ \emph{homogeneously} to a map $f\colon M\llbracket t_1,\ldots,t_n\rrbracket \to M\llbracket t_1,\ldots,t_n\rrbracket$ by putting $f(at_1^{i_1}\cdots t_n^{i_n}):=f(a)t_1^{i_1}\cdots t_n^{i_n}$ for all $a\in M$ and  $i_1,\ldots,i_n\in\mathbb{N}$. The case for binary maps is analogous.

\begin{definition}[Multi-parameter formal hom-associative deformation]\label{def:multi-param-hom-assoc} A \emph{multi-parameter}, or an \emph{$n$-parameter formal hom-associative deformation} of a hom-as\-so\-cia\-tive algebra $(M,\cdot_0,\alpha_0)$ over $R$, is a hom-associative algebra $(M\llbracket t_1,\ldots,t_n\rrbracket,\cdot_t,\alpha_t)$ over $R\llbracket t_1,\ldots, t_n\rrbracket$ where $n\in\mathbb{N}_{>0}$, $t:=(t_1,\ldots,t_n)$, and
\begin{equation*}
	\cdot_t=\sum_{i\in\mathbb{N}^n}\cdot_i t^i,\quad \alpha_t=\sum_{i\in\mathbb{N}^n}\alpha_i t^i.
\end{equation*}
Here, $i:=(i_1,\ldots,i_n)\in\mathbb{N}^n$ and $t^i:=t_1^{i_1}\cdots t_n^{i_n}$. Moreover, $\cdot_i\colon M\times M\to M$ is an $R$-bilinear map, extended homogeneously to an $R\llbracket t_1,\ldots,t_n\rrbracket$-bilinear map $\cdot_i\colon M\llbracket t_1,\ldots,t_n \rrbracket\times M\llbracket t_1,\ldots,t_n\rrbracket\to M\llbracket t_1,\ldots,t_n \rrbracket$. Similarly, $\alpha_i\colon M\to M$ is an $R$-linear map, extended homogeneously to an $R\llbracket t_1,\ldots,t_n\rrbracket$-linear map denoted by $\alpha_i\colon M\llbracket t_1,\ldots,t_n\rrbracket\to M\llbracket t_1,\ldots,t_n\rrbracket$.
\end{definition}

\begin{proposition}\label{prop:hom-assoc-deform}
$A_1^k$ is a multi-parameter formal deformation of $A_1$.	
\end{proposition}

\begin{proof}
Let $k=(k_0,k_p,k_{2p},\ldots,k_{Mp},0,0,0,\ldots)$ for some $M\in\mathbb{N}$ and put, for any $i\in\mathbb{N}$, $k_i=0$ unless $i=0,p,2p,\ldots,Mp$, or $i=1$, in which case $k_1=1_{K}$. By using the multinomial theorem, for an arbitrary monomial $y^mx^n$ where $m,n\in\mathbb{N}$,
\begin{align*}
&\alpha_k(y^mx^n)=(k_0+k_1y+k_2y^2+\cdots +k_{Mp}y^{Mp})^mx^n\\
&=\sum_{i_0+i_1+i_2+\cdots i_{Mp}=m}\frac{m!}{i_0!i_1!i_2!\cdots i_{Mp}!}\left(\prod_{j=0}^{Mp}(k_jy^j)^{i_j}\right) x^n\\
&=\sum_{i_0+i_1+i_2+\cdots i_{Mp}=m} \frac{m!}{i_0!i_1!i_2!\cdots i_{Mp}!}\left(\prod_{j=0}^{Mp}(y^j)^{i_j}\right)x^n \prod_{j=0}^{Mp} k_j^{i_j}\\
&= \sum_{i_0+i_1+i_2+\cdots i_{Mp}=m }  \frac{m!}{i_0!i_1!i_2!\cdots i_{Mp}!} \left(\prod_{j=0}^{Mp} y^{ji_j}\right) x^n \left( k_0^{i_0} k_1^{i_1}k_2^{i_2}\cdots k_{Mp}^{i_{Mp}} \right) \\
&= \sum_{i_0+i_p+i_{2p}+\cdots + i_{Mp} \leq m} \frac{m!}{i_0!i_{p}!i_{2p}!\cdots i_{Mp}! \cdot (m-i_0-i_p-i_{2p}-\ldots -i_{Mp})!}  \\
&\phantom{=}\; \cdot\left(\prod_{j=1}^{M} y^{pji_{jp}}\right)x^n \left( k_0^{i_0}k_p^{i_p}k_{2p}^{i_2p} \cdots k_{Mp}^{i_{Mp}} \right) \\
\end{align*}

Now define $t_1:=k_0$, $t_2:=k_p$, $t_3:=k_{2p},\ldots,t_{M+1}:=k_{Mp}$ and regard $t_1,\ldots,t_{M+1}$ as indeterminates of the formal power series $K\llbracket t_1,\ldots,t_{M+1}\rrbracket$ and $A_1\llbracket t_1,\ldots,t_{M+1}\rrbracket$. Then, by the above calculation, $\alpha_t$ is a formal power series in $t_1,\ldots,t_{M+1}$ where $t:=(t_1,\ldots,t_{M+1})$. For any specific element $q\in A_1$, $\alpha_t(q)$ will be a polynomial. Moreover, $\alpha_0=\mathrm{id}_{A_1}$. Next, we extend $\alpha_t$ linearly over $K\llbracket t_1,\ldots,t_{M+1}\rrbracket$ and homogeneously to all of $A_1\llbracket t_1,\ldots,t_{M+1}\rrbracket$. To define the multiplication $\cdot_t$ in $A_1\llbracket t_1,\ldots,t_{M+1}\rrbracket$, we first extend the multiplication $\cdot_0$ in $A_1$ homogeneously to a binary operation $\cdot_0\colon A_1\llbracket t_1,\ldots,t_n\rrbracket \times A_1\llbracket t_1,\ldots,t_n\rrbracket \to A_1\llbracket t_1,\ldots,t_n\rrbracket$ linear over $K\llbracket t_1,\ldots, t_n\rrbracket$ in both arguments. Then we compose $\alpha_t$ with $\cdot_0$, so that $\cdot_t:=\alpha_t\circ\cdot_0$. This is again a formal power series in $t_1,\ldots,t_{M+1}$, and hom-associativity follows from \autoref{prop:star-alpha-mult}. Hence we have a formal deformation $(A_1\llbracket t_1,\ldots,t_{M+1}\rrbracket, \cdot_t,\alpha_t)$ of $(A_1,\cdot_0,\alpha_0)$, where the latter is $A_1$ in the language of hom-associative algebras. 
\end{proof}

\begin{definition}[Multi-parameter formal hom-Lie deformation]\label{def:multi-param-hom-Lie} A \emph{multi-parameter}, or an \emph{$n$-parameter formal hom-Lie deformation} of a hom-Lie algebra $(M,[\cdot,\cdot]_0,\alpha_0)$ over $R$, is a hom-Lie algebra $(M\llbracket t_1,\ldots,t_n\rrbracket, [\cdot,\cdot]_t,\alpha_t)$ over $R\llbracket t_1,\ldots,t_n\rrbracket$ where $n\in\mathbb{N}_{>0}$, $t:=(t_1,\ldots,t_n)$, and 
\begin{equation*}
[\cdot,\cdot]_t=\sum_{i\in\mathbb{N}^n}[\cdot,\cdot]_it^i,\quad \alpha_t=\sum_{i\in\mathbb{N}^n}\alpha_i t^i.
\end{equation*}
Here, $i:=(i_1,\ldots,i_n)\in\mathbb{N}^n$ and $t^i:=t_1^{i_1}\cdots t_n^{i_n}$. Moreover,  $[\cdot,\cdot]_i\colon M\times M\to M$ is an $R$-bilinear map, extended homogeneously to an $R\llbracket t_1,\ldots,t_n\rrbracket$-bilinear map $[\cdot,\cdot]_i\colon M\llbracket t_1,\ldots,t_n \rrbracket\times M\llbracket t_1,\ldots,t_n\rrbracket\to M\llbracket t_1,\ldots,t_n \rrbracket$. Similarly, $\alpha_i\colon M\to M$ is an $R$-linear map, extended homogeneously to an $R\llbracket t_1,\ldots,t_n\rrbracket$-linear map denoted by $\alpha_i\colon M\llbracket t_1,\ldots,t_n\rrbracket\to M\llbracket t_1,\ldots,t_n\rrbracket$.
\end{definition}

\begin{proposition}\label{prop:hom-lie-deform}
The deformation of $A_1$ into $A_1^k$ induces a multi-parameter formal hom-Lie deformation of the Lie algebra of $A_1$ into the hom-Lie algebra of $A_1^k$, when using the commutator as bracket.	
\end{proposition}

\begin{proof}
Let $k=(k_0,k_p,k_{2p},\ldots,k_{Mp},0,0,0,\ldots)$. Then, by using the deformation $(A_1\llbracket t_1,\ldots,t_{M+1}\rrbracket, \cdot_t,\alpha_t)$ of $(A_1,\cdot_0,\alpha_0)$ in \autoref{prop:hom-assoc-deform}, we construct a hom-Lie algebra $(A_1\llbracket t_1,\ldots,t_{M+1}\rrbracket, [\cdot,\cdot]_t,\alpha_t)$ by using the commutator $[\cdot,\cdot]_t$ of the hom-associative algebra $(A_1\llbracket t_1,\ldots,t_{M+1}\rrbracket, \cdot_t,\alpha_t)$ as bracket. Indeed, by \autoref{prop:hom-assoc-commutator} this gives a hom-Lie algebra. We claim that this is also a formal hom-Lie deformation of the Lie algebra $(A_1,[\cdot,\cdot]_0,\alpha_0)$ where $[\cdot,\cdot]_0$ is the commutator in $A_1$, and $\alpha_0=\mathrm{id}_{A_1}$. Since $\alpha_t$ is the same map as in \autoref{prop:hom-assoc-deform}, we only need to verify that $[\cdot,\cdot]_t$ is a formal power series in $t_1,\ldots,t_n$, which when evaluated at $t=0$ gives the commutator in $A_1$. But this is immediate since $[\cdot,\cdot]_t=\alpha_t\circ[\cdot,\cdot]_0$.
\end{proof}

\section*{Acknowledgements}
P. Bäck would like to thank S.~A.~Lopes for the suggestion of studying the hom-associative Weyl algebras in prime characteristic.


\begin{thebibliography}{99}
\bibitem{BLO15}
G.~Benkart, S.~A.~Lopes, and M.~Ondrus,
\emph{Derivations of a parametric family of subalgebras of the Weyl algebra}, J. Algebra \textbf{424} (2015), pp. 46--97.

\bibitem{Bac20}
P.~B{\"a}ck,
\emph{Multi-parameter Formal Deformations of Ternary Hom-Nambu-Lie Algebras}, Dobrev V. (eds) Lie Theory and Its Applications in Physics. Springer Proceedings in Mathematics \& Statistics, vol 335. Springer, Singapore, 2020.

\bibitem{BR18}
P.~B{\"a}ck, J.~Richter,
\emph{Hilbert's basis theorem for non-associative and hom-associative Ore extensions}, \texttt{arXiv:1804.11304}.

\bibitem{BR20}
P.~B{\"a}ck, J.~Richter,
\emph{On the hom-associative Weyl algebras}, J. Pure Appl. Algebra \textbf{224}(9) (2020).

\bibitem{BRS18}
P.~B{\"a}ck, J.~Richter, S.~Silvestrov,	
\emph{Hom-associative Ore extensions and weak unitalizations},
Int. Electron. J. Algebra \textbf{24} (2018), pp. 174--194.

\bibitem{Cou95}
S.~C.~Coutinho,
\emph{A primer of algebraic D-modules}, 
Cambridge University Press, Cambridge, 1995.

\bibitem{Dix68}
J. Dixmier,
\emph{Sur les algèbres de Weyl},
Bull. Soc. Math. France \textbf{96} (1968), pp. 209–-242.

\bibitem{FG09}
Y.~Frégier, A.~Gohr, \emph{On unitality conditions for hom-associative algebras}, \texttt{arXiv:0904.4874}.

\bibitem{GW04}
K.~R.~Goodearl, R.~B.~Warfield,
\emph{An Introduction to Noncommutative Noetherian Rings},
Cambridge University Press, Cambridge U.~K., 2004.

\bibitem{HLS03} 
J.~T.~Hartwig, D.~Larsson, S.~D.~Silvestrov,
\emph{Deformations of Lie algebras using $\sigma$-derivations},
J. Algebra \textbf{295} (2006), pp. 314--361.

\bibitem{KBK07}
A.~Kanel-Belov, M.~Kontsevich,
\emph{The Jacobian conjecture is stably equivalent to the Dixmier conjecture},
Mosc. Math. J. \textbf{7}(2) (2007), pp. 209--218.

\bibitem{Luc78}
E.~Lucas,
\emph{Sur les congruences des nombres eulériens et des coefficients différentiels des fonctions trigonométriques suivant un module premier},
Bull. Soc. Math. France \textbf{6} (1878), pp. 49--54.

\bibitem{Mak84}
L.~Makar-Limanov, 
\emph{On automorphisms of Weyl algebra},
Bull. Soc. Math. France \textbf{112} (1984), pp. 359--363.

\bibitem{MS08}
A.~Makhlouf, S.~D.~Silvestrov,
\emph{Hom-algebra structures},
J. Gen. Lie Theory Appl. \textbf{2} (2008), pp. 51--64.

\bibitem{MS10a}
A.~Makhlouf, S.~Silvestrov,
\emph{Notes on 1-parameter formal deformations of Hom-associative and Hom-Lie algebras},
 Forum Math. \textbf{22} (2010), pp. 715--739.

\bibitem{MR01}
J.~C.~McConnell, J.~C.~Robson, L.~W.~Small, 
\emph{Noncommutative Noetherian Rings}, 
American Mathematical Society, Providence, R.~I., 2001[1987].

\bibitem{NOR18}
P.~Nystedt, J.~ \"Oinert, and J.~Richter,
\emph{Non-associative Ore extensions},
Isr. J. Math. \textbf{224}(1) (2018), pp. 263--292.

\bibitem{Ore33}
O.~Ore,
\emph{Theory of Non-Commutative Polynomials},
Ann. of Math. \textbf{34} (1933), pp. 480--508.

\bibitem{Rev73}
M.~P.~Revoy,  
\emph{Algèbres  de  Weyl  en  caractéristique $p$},
C. R. Acad. Sci. Paris  Sér.  A \textbf{276} (1973), pp. 225--228.

\bibitem{Sri61}
R.~Sridharan, 
\emph{Filtered Algebras and Representations of Lie Algebras}, 
Trans. Amer. Math.Soc. \textbf{100}(3) (1961), pp. 530-–550.

\bibitem{Tsu05}
Y.~Tsuchimoto,
\emph{Endomorphisms of Weyl algebra and $p$-curvatures}, Osaka J. Math. \textbf{42}(2) (2005), pp. 435--452.

\bibitem{Tsu03}
Y.~Tsuchimoto,
\emph{Preliminaries on Dixmier conjecture}
Mem. Fac. Sci. Kochi Univ. (Math.) \textbf{24} (2003).

\bibitem{Yau09} 
D.~Yau, 
\emph{Hom-algebras and Homology},
J. Lie Theory \textbf{19}(2) (2009), pp. 409--421.

\end{thebibliography}
\end{document}